\documentclass{article} % For LaTeX2e
\usepackage{iclr2025_conference,times}

\usepackage{amsmath,amsfonts,bm}
\def\eqref#1{equation~\ref{#1}}
\def\1{\bm{1}}

\def\va{{\bm{a}}}
\def\vb{{\bm{b}}}

\def\ve{{\bm{e}}}

\def\vh{{\bm{h}}}

\def\vu{{\bm{u}}}
\def\vv{{\bm{v}}}
\def\vw{{\bm{w}}}
\def\vx{{\bm{x}}}
\def\vy{{\bm{y}}}
\def\vz{{\bm{z}}}

\def\mA{{\bm{A}}}

\def\mF{{\bm{F}}}

\def\mI{{\bm{I}}}

\def\mQ{{\bm{Q}}}

\def\mU{{\bm{U}}}

\DeclareMathAlphabet{\mathsfit}{\encodingdefault}{\sfdefault}{m}{sl}
\SetMathAlphabet{\mathsfit}{bold}{\encodingdefault}{\sfdefault}{bx}{n}

\def\gO{{\mathcal{O}}}

\def\sB{{\mathbb{B}}}
\def\sC{{\mathbb{C}}}

\def\sR{{\mathbb{R}}}

\newcommand{\R}{\mathbb{R}}

\DeclareMathOperator*{\argmin}{arg\,min}

\usepackage{amsthm}
\theoremstyle{plain}
\newtheorem{thm}{Theorem}[section]
\newtheorem{dfn}{Definition}[section]

\newtheorem{lem}{Lemma}[section]
\newtheorem{asm}{Assumption}[section]
\newtheorem{remark}{Remark}[section]
\newtheorem{cor}{Corollary}[section]

\usepackage{amsmath}
\usepackage{hyperref}
\usepackage{url}
\usepackage{algorithmic,algorithm}
\usepackage{nicefrac}
\usepackage{cancel}
\usepackage{graphicx} 
\usepackage{booktabs}
\usepackage{makecell}
\usepackage{multirow}
\usepackage{amssymb}

\title{Second-Order Min-Max Optimization \\ with
Lazy Hessians}

\author{Lesi Chen $^*$ \\
IIIS, Tsinghua University \& Shanghai Qizhi Institute \\
\texttt{chenlc23@mails.tsinghua.edu.cn} \\
\And
Chengchang Liu $^*$ \\
The Chinese University of Hong Kong \\
\texttt{7liuchengchang@gmail.com} \\
\AND
Jingzhao Zhang $^\dagger$ \\
IIIS, Tsinghua University \& Shanghai Qizhi Institute \& Shanghai AI Lab \\
\texttt{jingzhaoz@mail.tsinghua.edu.cn}
}

\iclrfinalcopy % Uncomment for camera-ready version, but NOT for submission.
\begin{document}
\maketitle
\begin{abstract}
This paper studies second-order methods for convex-concave minimax optimization. 
\citet{monteiro2012iteration} proposed a method to solve the problem with an optimal iteration complexity of
$\gO(\epsilon^{-3/2})$ to find an $\epsilon$-saddle point.  
%Over the past decade, researchers proposed various methods with the same convergence rate.
% Recent works~\cite{adil2022optimal,lin2022perseus,bullins2022higher} have achieved an optimal iteration complexity of
However, it is unclear whether the
computational complexity, $\gO((N+ d^2) d \epsilon^{-2/3})$, can be improved. In the above, we follow \citet{doikov2023second} and assume the complexity of obtaining a first-order oracle as
$N$ and the complexity of obtaining a second-order oracle as $dN$. In this paper, we show that the computation cost can be reduced by reusing Hessian across iterations.
Our methods take the overall computational complexity of $ \tilde{\gO}( (N+d^2)(d+ d^{2/3}\epsilon^{-2/3}))$, which improves those of the previous methods by a factor of $d^{1/3}$. 
Furthermore, we generalize our method to strongly-convex-strongly-concave minimax problems and establish the complexity of $\tilde{\gO}((N+d^2) (d + d^{2/3} \kappa^{2/3}) )$ when the condition number of the problem is $\kappa$, enjoying a similar speedup upon the state-of-the-art method. 
Numerical experiments on both real and synthetic datasets also verify the efficiency of our method. 
\end{abstract}
{
\renewcommand\thefootnote{$^*$}
\footnotetext{Equal contributions.}
\renewcommand\thefootnote{$^\dagger$}
\footnotetext{The corresponding author.}
}

\section{Introduction}
We consider the following minimax optimization problem:
\begin{align} \label{prob:min-max}
    \min_{\vx \in \R^{d_x}} \max_{\vy \in \R^{d_y}} f(\vx,\vy),
\end{align}
where we suppose $f(\vx,\vy)$ is (strongly-)convex in $\vx$ and (strongly-)concave in $\vy$. 
% In the absence of convex-convexity,
% this problem is intractable in general, as
% any first-order algorithm may require an exponential number of gradient queries to achieve even a local min-max point \citep{daskalakis2021complexity}.
% Therefore we stick to the \textit{convex-concave} setting in this paper.
This setting covers many useful applications, including functionally constrained optimization~\citep{xu2020primal}, game theory~\citep{von1947theory}, robust optimization~\citep{ben2009robust},  fairness-aware machine learning~\citep{zhang2018mitigating}, reinforcement learning~\citep{du2017stochastic,wang2017primal,paternain2022safe,wai2018multi}, decentralized optimization~\citep{kovalev2021lower,kovalev2020optimal}, AUC maximization~\citep{ying2016stochastic,hanley1982meaning,yuan2021large}. 

%Under the convex-concave assumption, problem (\ref{prob:min-max}) is also known as the \textit{saddle point problem (SPP)}, since its solution is the saddle point $(\vx^*,\vy^*)$ such that
% \begin{align*}
%     f(\vx^*,\vy ) \le f(\vx, \vy) \le f(\vx, \vy^*), \quad \forall \vx \in \R^{d_x}, ~\vy \in \R^{d_y}.
% \end{align*}

First-order methods are widely studied for this problem. Classical algorithms include ExtraGradient (EG)~\citep{korpelevich1976extragradient,nemirovski2004prox}, Optimistic Gradient Descent Ascent (OGDA)~\citep{popov1980modification,mokhtari2020unified,mokhtari2020convergence}, Hybrid Proximal Extragradient (HPE)~\citep{monteiro2010complexity}, and Dual Extrapolation (DE)~\citep{nesterov2006solving,nesterov2007dual}.
When the gradient of $f(\,\cdot\,,\,\cdot\,)$ is $L$-Lipschitz continuous,
these methods achieve the rate of $\gO(\epsilon^{-1})$ under the convex-concave (C-C) setting and the rate of $\gO((L/\mu)\log(\epsilon^{-1}))$ when $f(\,\cdot\,,\,\cdot\,)$ is $\mu$-strongly convex in $\vx$ and $\mu$-strongly-concave in $\vy$ (SC-SC) for $\mu>0$.
They are all optimal in C-C and SC-SC settings due to the lower bounds reported by \citep{nemirovskij1983problem,zhang2022lower}.

% These methods have the same convergence rates: for convex-concave problems, they can find an $\epsilon$-saddle point {\color{red}(c.f. )} within $\gO(1/\epsilon)$ first-order oracle calls; for strongly-convex-strong-concave problem, the oracle complexity is $\gO(\kappa \log(\nicefrac{1}{\epsilon}))$, where $\kappa$ is the condition number. 
% These rates are optimal due to the existence of matching lower bounds: the $\Omega(1/\epsilon)$ lower bound for convex-concave problems can be found in \citet{nemirovskij1983problem}, while the $\Omega(\kappa \log(\nicefrac{1}{\epsilon}))$ lower bound for strongly-convex-strongly-concave problems is established by \citet{zhang2022lower}.

Second-order methods usually lead to faster rates than first-order methods when the Hessian of $f(\,\cdot\,,\,\cdot\,)$ is $\rho$-Lipschitz continuous.
% To achieve a better iteration complexity over the first-order methods for solving minimax optimization, it is common to consider second-order methods, which generally exhibit superior convergence behaviors than first-order methods. 
A line of works~\citep{nesterov2006solving,huang2022cubic} extended the celebrated Cubic Regularized Newton (CRN)~\citep{nesterov2006cubic}
method to minimax problems with local superlinear convergence rates and global convergence guarantee. 
However, the established global convergence rates of  $\gO(\epsilon^{-1})$ by~\citet{nesterov2006solving} and $\gO((L\rho/\mu^2) \log (\epsilon^{-1}))$ by~\citet{huang2022cubic} under C-C and SC-SC conditions are no better than the optimal first-order methods. 
Another line of work
%~\citep{monteiro2012iteration,adil2022optimal,lin2022perseus,nesterov2023high,huang2022approximation,bullins2022higher,jiang2022generalized} 
generalizes the optimal first-order methods to higher-order methods.
\citet{monteiro2012iteration} proposed the Newton Proximal Extragradient (NPE) method with a global convergence rate of $\gO(\epsilon^{-2/3} \log \log (\epsilon^{-1}))$ under the C-C conditions. This result nearly matches the lower bounds~\citep{adil2022optimal,lin2022perseus}, except an additional $\gO(\log \log(\epsilon^{-1}))$ factor which is caused by the implicit binary search at each iteration. 
\citet{bullins2022higher,adil2022optimal,huang2022approximation,lin2022explicit} provided a simple proof of NPE motivated by the EG analysis and showed that replacing the quadratic regularized Newton step with the cubic regularized Newton (CRN) step in NPE achieves the optimal second-order oracle complexity of $\gO( \epsilon^{-2/3})$.
Recently, \citet{alves2023search} proposed 
a search-free NPE method to achieve the optimal second-order oracle complexity with pure quadratic regularized Newton step based on ideas from homotopy.
Over the past decade, researchers also proposed various second-order methods, in addition to the NPE framework, that achieve the same convergence rate, such as the second-order extensions of OGDA \citep{jiang2022generalized,jiang2024adaptive}   (which we refer to as OGDA-2) and DE~\citep{lin2022perseus} (they name the method Persesus).
The results for C-C problems can also be extended to SC-SC problems, where \citet{jiang2022generalized} proved the OGDA-2 can converge at the rate of $\gO( (\rho/\mu)^{2/3} + \log \log (\epsilon^{-1}) )$, and \citet{huang2022approximation} proposed the ARE-restart with the rate of $\gO((\rho/\mu)^{2/3} \log \log (\epsilon^{-1}))$.

Although the aforementioned second-order methods~\cite{adil2022optimal,lin2022perseus,lin2022explicit,jiang2022generalized,monteiro2012iteration} enjoy an improved convergence rate over the first-order methods and have achieved optimal iteration complexities, they require querying one new Hessian at each iteration and solving a matrix inversion problem at each Newton step, which leads to a $\gO(d^{3})$ computational cost per iteration. This becomes the main bottleneck that limits the applicability of second-order methods. \citet{liu2022quasi} proposed quasi-Newton methods for saddle point problems that access one Hessian-vector product instead of the exact Hessian for each iteration. The iteration complexity is $\gO(d^2)$ for quasi-Newton methods. However, their methods do not have a global convergence guarantee under general (S)C)-(S)C conditions. \citet{jiang2023online} proposed an online-learning guided Quasi-Newton Proximal Extragradient (QNPE) algorithm, but their method relies on more complicated subroutines than classical Newton methods. Although the oracle complexity of QNPE is strictly better than the optimal first-order method EG, their method is worse in terms of total computational complexity.

In this paper, we propose a computation-efficient second-order method, which we call LEN (Lazy Extra Newton method). 
In contrast to all existing second-order methods or quasi-Newton methods for minimax optimization problems that always access new second-order information for the coming iteration, LEN reuses the second-order information from past iterations. 
Specifically, LEN solves a cubic regularized sub-problem using the Hessian from the snapshot point that is updated every $m$ iteration, then conducts an extra-gradient step by the gradient from the current iteration.
We provide a rigorous theoretical analysis of LEN to show it maintains fast global convergence rates and improves the (near)-optimal second-order methods~\citep{monteiro2012iteration} in terms of the overall computational complexity.
We summarize our contributions as follows (also see Table \ref{tab:res}).

\vspace{-0.2cm}
\begin{itemize}
    \item 
When the object function $f(\,\cdot\,,\,\cdot\,)$ is convex in $\vx$ and concave in $\vy$, we propose LEN and prove that it finds an $\epsilon$-saddle point in $\mathcal{O}(m^{2/3}\epsilon^{-2/3})$ iterations. Under Assumption \ref{asm:arith-cmp}, where 
    the complexity of calculating $\mF(\vz)$ is  $N$ and the complexity of calculating $\nabla \mF(\vz)$ is $d N$, the optimal choice is $m = \Theta(d)$. In this case, LEN
    only requires a computational complexity of $\tilde \gO( (N+d^2)(d+ d^{2/3}\epsilon^{-2/3}) )$, which is strictly better than $\gO((N+d^2) d \epsilon^{-2/3})$ for the existing optimal second-order methods by a factor of $d^{1/3}$. 
    % which can recover the optimal rates~\cite{adil2022optimal,lin2022explicit,lin2022perseus,huang2022approximation} when choose $m=1$.
    \vspace{-0.02cm}
    \item 
    When the object function $f(\,\cdot\,,\,\cdot\,)$ is $\mu$-strongly-convex in $\vx$ and $\mu$-strongly-concave in $\vy$, we apply the restart strategy on LEN and propose LEN-restart. We prove the algorithm can find an $\epsilon$-root with 
    $\tilde \gO( (N+d^2) (d + d^{2/3} (\rho/\mu)^{2/3} ))$ computational complexity, where $\rho$ means the Hessian of $f(\cdot,\cdot)$ is $\rho$ Lipschitz-continuous. Our result
    % By taking $m=\Theta(d)$, LEN-restart takes the arithmetic oracle complexity of $\mathcal{O}(d^{2/3}\log\log(\epsilon^{-1}))$,
    is strictly better than the $\tilde \gO( (N+d^2) d (\rho/\mu)^{2/3} )$ in prior works.
\end{itemize} 
%It is worth noting that, our result is the first to achieve a polynomial improvement in running time since  \citet{monteiro2012iteration}.
\begin{table*}[t]
  \caption{
  We compare the required computational complexity to achieve an $\epsilon$-saddle point of
the proposed LEN with the optimal choice $m=\Theta(d)$ and other existing algorithms on both convex-concave (C-C) and strongly-convex-strongly-concave (SC-SC) problems. Here, $d = d_x+d_y$ is the dimension of the problem. 
We assume the gradient is $L$-Lipschitz continuous for EG and the Hessian is $\rho$-Lipschitz continuous for others.
We count each gradient oracle call with $N$ computational complexity, and each Hessian oracle with $dN$ computational complexity. 
%\jz{If possible, we could add a column and use the format as in Yin-tat-Lee's paper(n SO + m HO + d / e etc) so that we don't have to assume the N-Nd computation cost.  }
% We remark in ``*''  that the superlinear rates of quasi-Newton methods~\citep{liu2022quasi} hold only for quadratics or locally for non-quadratics and require additional assumptions on the smoothness of the objectives.
    }
        \label{tab:res}
    \centering
    \begin{tabular}{c c c }
    \hline 
    Setup & Method  & Computational Cost \\ 
     \hline \hline  \addlinespace
    %\addlinespace & EG~\citep{korpelevich1976extragradient}  & $\fO(1/\epsilon)$ & $\fO(d)$  \\  
 & EG~\citep{korpelevich1976extragradient} & $\gO( (N+d) \epsilon^{-1} )$ \\ \addlinespace
 & NPE~\citep{monteiro2012iteration}      & $\tilde\gO ( (N+d^2) d \epsilon^{-2/3} ) $ \\ \addlinespace
C-C & search-free NPE~\citep{alves2023search} & $\gO( (N+d^2) d \epsilon^{-2/3})$ \\ \addlinespace 
 % C-C & NPE with CRN~\citep{bullins2022higher} & $ \tilde \gO((N+d^2) d \epsilon^{-2/3} ) $ \\ \addlinespace
 % & Perseus~\citep{lin2022perseus} & $ \tilde \gO((N+d^2) d \epsilon^{-2/3} ) $ \\ \addlinespace 
 & OGDA-2~\citep{jiang2022generalized} & $\gO( (N+d^2) d \epsilon^{-2/3})$ \\ \addlinespace
    & LEN (Theorem \ref{thm:LEN-CC-complexity})    & {$\tilde \gO( (N+d^2) (d + {\color{blue} d^{2/3}} \epsilon^{-2/3}) ) $} \\ \addlinespace
\hline \addlinespace
& EG~\citep{korpelevich1976extragradient}  & $\tilde \gO( (N+d) (L/\mu))$  \\ \addlinespace 
% & Quasi-Newton~\citep{liu2022quasi} &  $\tilde \gO( (N+d^2) d^{1/2})$ for quadratics* \\ 
% \addlinespace
% & Perseus-restart~\citep{lin2022perseus} & $\gO(d\log (\nicefrac{1}{\epsilon}))$  
% \\ \vspace{-2mm}
% SC-SC  \\ 
& OGDA-2~\citep{jiang2022generalized} & $ \gO((N+d^2) d (\rho/\mu)^{2/3}) $ \\ \addlinespace
SC-SC & ARE-restart~\citep{huang2022approximation}  & $\tilde \gO((N+d^2) d (\rho/\mu))^{2/3}) $\\ \addlinespace
 & Perseus-restart~\citep{lin2022perseus} & $\tilde \gO((N+d^2) d (\rho/\mu)^{2/3}) $ \\ \addlinespace
& LEN-restart (Corollary \ref{thm:LEN-SCSC-complexity})     &{$\tilde \gO( (N+d^2) (d+{\color{blue} d^{2/3}} (\rho/\mu)^{2/3} ))$ } \\ \addlinespace
    \hline
    \end{tabular}
\end{table*}
\paragraph{Notations.} Throughout this paper, $\log$ is base $2$ and $\log_+(\,\cdot\,) := 1 + \log(\,\cdot\,)$. We use $\Vert \cdot \Vert$ to denote the spectral norm and the Euclidean norm of matrices and vectors, respectively.
We denote $\pi(t) = t-  (t \mod m )  $ where $m\in\mathbb{N_+}$.

\section{Related Works and Technical Challenges} \label{sec:related}

\paragraph{Lazy Hessian in minimization problems.} The idea of reusing Hessian was initially presented by \citet{shamanskii1967modification} and later incorporated into the Levenberg-Marquardt method, the Damped Newton method, and the proximal Newton method~\citep{fan2013shamanskii,lampariello2001global,wang2006further,adler2020new}. 
However, the explicit advantage of lazy Hessian update over ordinary Newton (-type) update was not discovered until the recent work of~\citep{doikov2023second, chayti2023unified}. 
They applied the following lazy Hessian update on cubic regularized Newton (CRN) methods~\citep{nesterov2006cubic}:
\begin{align} \label{eq:lazy-CRN}
    \vz_{t+1} = \argmin_{\vz \in \sR^d} \left\{ \langle \mF(\vz_t), \vz - \vz_t \rangle + \frac{1}{2} \langle 
 \nabla \mF({\color{blue}\vz_{\pi(t)}}) (\vz - \vz_t) , \vz- \vz_t \rangle + \frac{M}{6} \Vert \vz - \vz_t \Vert^3  \right\},
\end{align}
where $M\geq0$ and $\mF:\sR^d \rightarrow \sR^d$ is the gradient field of a convex function.
They establish the convergence rates of $\mathcal{O}(\sqrt{m}  \epsilon^{-3/2})$ for nonconvex optimization~\citep{doikov2023second}, and $\mathcal{O}(\sqrt{m} \epsilon^{-1/2})$ for convex optimization~\citep{chayti2023unified} respectively.  Such rates lead to the total computational cost of $\tilde \gO((N+d^2) (d+ \sqrt{d} \epsilon^{-3/2}))$ and $\tilde \gO((N+d^2) (d+\sqrt{d} \epsilon^{-1/2}))$ by setting $m = \Theta(d)$, which strictly improve the result by classical CRN methods by a factor of $\sqrt{d}$ in both setups.

We have also observed that the idea of the ``lazy Hessian" is widely used in practical second-order methods. KFAC~\citep{martens2015optimizing,grosse2016kronecker} approximates the Fisher information matrix and uses an exponential moving average (EMA) to update the estimate of the Fisher information matrix, which can be viewed as a soft version of lazy update.
Sophia~\citep{liu2023sophia} estimates a diagonal Hessian matrix as a pre-conditioner, which is updated in a lazy manner to reduce the complexity.
C2EDEN~\citep{liu2023communication} atomizes the communication of local Hessian in several consecutive iterations, which also benefits from the idea of lazy updates.

\paragraph{Challenge of using lazy Hessian updates in minimax problems.}
In comparison to previous work on lazy Hessian, our LEN and LEN-restart methods demonstrate the advantage of using lazy Hessian for a broader class of optimization problems, the \textit{minimax} problems.
Our analysis differs from the ones in \citet{doikov2023second,chayti2023unified}.
Their methods only take a lazy CRN update (\ref{eq:lazy-CRN}) at each iteration, 
which makes it easy to bound the error of lazy Hessian updates using Assumption \ref{asm:prob-lip-hes} and the triangle inequality in the following way: 
\begin{align*}
 \| \nabla \mF(\vz_t)-\nabla \mF(\vz_{\pi(t)})\|\leq \rho \|\vz_{\pi(t)}-\vz_t\|\leq \rho \sum_{i=\pi(t)}^{t-1}\|\vz_{i}-\vz_{i+1}\|.   
\end{align*}
% the error term of using the lazy Hessian $\|\vz_{\pi(t)}-\vz_t\|$ can be easily controlled by the progress of consecutive points generated by the lazy CRN updates $\{\|\vz_{i+1}-\vz_i\|\}_{i=\pi(t)}^{t}$ according to the Lipschitz continuity of Hessian 

Our method, on the other hand, not only takes a lazy (regularized) Newton update but also requires an extra gradient step (Line \ref{line:extra} in Algorithm \ref{alg:LEN}). Thus, the summation of  Newton progress $ \sum_{i=\pi(t)}^{t-1} \Vert \vz_{i+1/2}-\vz_i \Vert $ cannot directly bound the error term $\|\vz_t-\vz_{\pi(t)}\|$ introduced by the lazy Hessian update.
Moreover, for minimax problems the matrix $\nabla \mF(\vz_{\pi(t)})$ is no longer symmetric, which leads to different analysis and implementation of sub-problem solving (Section \ref{sec:imple}). We refer the readers to Section~\ref{sec:alg} for more detailed discussions.
% Our final algorithm has overcome 

\section{Preliminaries} \label{sec:pre}

In this section, we introduce the notation and basic assumptions used in our work.
We start with several standard definitions for Problem (\ref{prob:min-max}).
\begin{dfn}
We call a function $f(\vx,\vy): \sR^{d_x} \times \sR^{d_y} \rightarrow \sR$ has $\rho$-Lipschitz Hessians if we have
\begin{align*}
       \Vert \nabla^2 f(\vx, \vy) - \nabla^2 f(\vx',\vy') \Vert \le \rho \left \Vert \begin{bmatrix}
           \vx - \vx' \\
           \vy - \vy' 
       \end{bmatrix} \right \Vert, \quad \forall (\vx,\vy), (\vx', \vy') \in \sR^{d_x} \times \sR^{d_y}.
\end{align*}
\end{dfn}

% For analyzing second-order methods, we typically assume the objective has Lipschitz continuous Hessians.
% We also give the formal definition of convex-concavity.

\begin{dfn}
   A differentiable function $f(\cdot,\cdot)$ is $\mu$-strongly-convex-$\mu$-strongly-concave for some $\mu>0$ if
   \begin{align*}
       f(\vx', \vy) &\ge f(\vx,\vy) + (\vx' - \vx)^\top \nabla_x f(\vx,\vy) + \frac{\mu}{2} \Vert \vx - \vx' \Vert^2, \quad \forall  \vx',\vx \in \sR^{d_x}, \vy \in \sR^{d_y}; \\
       f(\vx,\vy')  &\le f(\vx,\vy) + (\vy' - \vy)^\top \nabla_y f(\vx,\vy) - \frac{\mu}{2} \Vert \vy - \vy' \Vert^2, \quad \forall \vy',\vy \in \sR^{d_y},  \vx \in \sR^{d_x}.
   \end{align*}
We say $f$ is convex-concave if $\mu=0$.
\end{dfn}
% \begin{dfn}
%    A differentiable function $f$ is convex-concave if
%    \begin{align*}
%        f(\vx', \vy) &\ge f(\vx,\vy) + (\vx' - \vx)^\top \nabla_x f(\vx,\vy), \quad \forall  \vx',\vx \in \sR^{d_x}, \vy \in \sR^{d_y} \\
%        f(\vx,\vy')  &\le f(\vx,\vy) + (\vy' - \vy)^\top \nabla_y f(\vx,\vy), \quad \forall \vy',\vy \in \sR^{d_y},  \vx \in \sR^{d_x}.
%    \end{align*}
% \end{dfn}
We are interested in finding a saddle point of Problem (\ref{prob:min-max}), formally defined as follows.

\begin{dfn}
   We call a point $(\vx^*,\vy^*) \in \sR^{d_x} \times  \sR^{d_y}$ a saddle point of a function $f(\cdot,\cdot)$ if we have
   \begin{align*}
       f(\vx^*,\vy ) \le f(\vx^*, \vy^*) \le f(\vx, \vy^*), \quad \forall \vx \in \R^{d_x}, ~\vy \in \R^{d_y}.
   \end{align*}
\end{dfn}

Next, we introduce all the assumptions made in this work. In this paper, we focus on Problem (\ref{prob:min-max}) that satisfies the following assumptions.

\begin{asm}
\label{asm:prob-lip-hes}
    We assume the function $f(\cdot,\cdot)$ is twice continuously differentiable, has $\rho$-Lipschitz continuous Hessians, and has at least one saddle point $(\vx^*,\vy^*)$.
\end{asm}

We will study convex-concave problems and strongly-convex-strongly-concave problems.
\begin{asm}[C-C setting] \label{asm:prob-cc}
    We assume the function $f(\cdot,\cdot)$ is convex in $\vx$ and concave in $\vy$.
% Further, we assume that there exists at least one saddle point $(\vx^*,\vy^*)$ for function $f$. 
%If $f$ is further $\mu$-strongly-convex-$\mu$-strongly-concave, we call it the SC-SC setting. Otherwise, we call it the C-C setting. 
\end{asm}
\begin{asm}[SC-SC setting] \label{asm:prob-scsc}
    We assume the function $f(\cdot,\cdot)$ is $\mu$-strongly-convex-$\mu$-strongly-concave. We denote the condition number as $\kappa:=\rho/\mu$.
\end{asm}
% For the SC-SC case, .

We let $d:=d_x+d_y$ and denote the aggregated variable $\vz := (\vx,\vy) \in \sR^d$. We also denote the GDA field of $f$ and its Jacobian as
\begin{align}
\begin{split}
    \mF(\vz) :=
\begin{bmatrix}
            \nabla_x f(\vx,\vy) \\
            - \nabla_y f(\vx, \vy) 
             \end{bmatrix}, \quad   \nabla \mF(\vz):= \begin{bmatrix}
        \nabla_{xx}^2 f(\vx,\vy) & \nabla_{xy}^2 f(\vx,\vy) \\
        -\nabla_{yx}^2 f(\vx, \vy) & - \nabla_{yy}^2 f(\vx,\vy) 
    \end{bmatrix}. 
\end{split}
\end{align}
The GDA field of $f(\cdot,\cdot)$ has the following properties.
\begin{lem}[Lemma 2.7 \citet{lin2022explicit}] \label{lem:op}
%  Let $d= d_x +d_y$ and denote the aggregated variable $\vz = (\vx,\vy) \in \sR^d$. Denote the GDA field of $f$ and its Jacobian as
% \begin{align} \label{eq:GDA-field}
% \begin{split}
%     \mF(\vz) = 
% \begin{bmatrix}
%             \nabla_x f(\vx,\vy) \\
%             - \nabla_y f(\vx, \vy) 
%              \end{bmatrix}, \quad   \nabla \mF(\vz)= \begin{bmatrix}
%         \nabla_{xx}^2 f(\vx,\vy) & \nabla_{xy}^2 f(\vx,\vy) \\
%         -\nabla_{yx}^2 f(\vx, \vy) & - \nabla_{yy}^2 f(\vx,\vy) 
%     \end{bmatrix}. 
% \end{split}
% \end{align}
 Under Assumptions \ref{asm:prob-lip-hes} and \ref{asm:prob-cc},  we have
 \begin{enumerate}
     \item $\mF$ is monotone, \textit{i.e.} $\langle \mF(\vz) - \mF(\vz'), \vz - \vz' \rangle \ge 0, ~\forall \vz,\vz' \in \sR^d $. 
     \item $\nabla \mF$ is $\rho$-Lipschitz continuous, \textit{i.e.} $\Vert \nabla \mF(\vz) - \nabla \mF(\vz') \Vert \le \rho \Vert \vz - \vz' \Vert,~\forall \vz, \vz' \in \sR^d$.
     \item $\mF(\vz^*) = 0$ if and only if $\vz^* = (\vx^*,\vy^*)$ is a saddle point of function $f(\cdot,\cdot)$.
 \end{enumerate}  
 Furthermore, if Assumption~\ref{asm:prob-scsc} holds, we have $\mF(\cdot)$ is $\mu$-strongly-monotone, \textit{i.e.} 
 \begin{align*}
     \langle \mF(\vz) - \mF(\vz'), \vz - \vz' \rangle \ge \mu\|\vz-\vz'\|^2, ~\forall \vz,\vz' \in \sR^d.
 \end{align*}
\end{lem}

For the C-C case, the commonly used optimality criterion is the following restricted gap.

\begin{dfn}[\citet{nesterov2007dual}] \label{dfn:gap}
Let $\sB_{\beta}(\vw)$ be the ball centered at $\vw$ with radius $\beta$. Let $(\vx^*,\vy^*)$ be a saddle point of function $f$.  For a given  point $(\hat \vx, \hat \vy)$, we let $\beta$ sufficiently large such that it holds
\begin{align*}
    \max \left\{ \Vert \hat \vx - \vx^* \Vert, ~ \Vert \hat \vy - \vy^*   \Vert \right\} \le \beta,
\end{align*}
we define the restricted gap function as 
    \begin{align*}
        {\rm Gap}(\hat \vx, \hat \vy ; \beta) := \max_{\vy \in \sB_{\beta}(\vy^*)} f(\hat \vx, \vy) -  \min_{\vx \in \sB_{\beta}(\vx^*)} f(\vx, \hat \vy),
    \end{align*}
We call $(\hat \vx, \hat \vy)$ an $\epsilon$-saddle point if $ {\rm Gap}(\hat \vx, \hat \vy ; \beta) \le \epsilon$ and $\beta = \Omega( \max \{ \Vert \vx_0 - \vx^* \Vert, \Vert \vy_0 - \vy^* \Vert \})$.
\end{dfn}

For the SC-SC case, we use the following stronger notion.

\begin{dfn}
Suppose that Assumption \ref{asm:prob-scsc} holds.
Let $\vz^* = (\vx^*,\vy^*)$ be the unique saddle point of function $f$. We call $\hat \vz = (\hat \vx, \hat \vy)$ an $\epsilon$-root if $\Vert \hat \vz - \vz^* \Vert \le \epsilon$.
\end{dfn}

Most previous works only consider the complexity metric as the number of oracle calls, where an oracle takes a point $\vz \in \sR^d$ as the input and returns a tuple $(\mF(\vz), \nabla \mF(\vz))$ as the output. 
The existing algorithms~\citep{monteiro2012iteration,bullins2022higher,adil2022optimal,lin2022explicit} have achieved 
optimal complexity regarding the number of oracle calls. In this work, we focus on the computational complexity of the oracle. More specifically, we distinguish between the different computational complexities of calculating the Hessian matrix $\nabla \mF(\vz)$ and the gradient $\mF(\vz)$. Formally, we make the following assumption as \citet{doikov2023second}.

\begin{asm} \label{asm:arith-cmp}
    We count the computational complexity of computing $\mF(\,\cdot\,)$ as $N$ and the computational complexity of 
    $\nabla \mF(\,\cdot\,)$ as $N d$.
\end{asm}
\begin{remark}
Assumption~\ref{asm:arith-cmp} supposes the cost of computing $\nabla \mF(\,\cdot\,)$ is $d$ times that of computing $\mF(\,\cdot\,)$. It holds in many practical scenarios as one Hessian oracle can be computed 
via $d$ Hessian-vector products on standard basis vectors $\ve_1, \cdots, \ve_d$, and one Hessian-vector product oracle is typically as expensive as one gradient oracle~\citep{wright2006numerical}:
\begin{enumerate}
    \item When the computational graph of $f$ is obtainable, both $\mF(\vz)$ and $\nabla \mF(\vz) \vv$ can be computed using automatic differentiation with the same cost for any $\vz, \vv \in \sR^d$.
    \item When $f$ is a black box function, we can estimate the Hessian-vector $\nabla \mF(\vz) \vv  $ via the finite-difference $ \vu_{\delta}(\vz;\vv) = \frac{1}{\delta} ( \mF(\vz + \delta \vv) - \mF(\vz - \delta \vv) ) $ and we have $\lim_{\delta \rightarrow 0} \vu_{\delta}(\vz;\vv) = \nabla \mF(\vz) \vv$ under mild conditions on $\mF(\,\cdot\,)$.
\end{enumerate}
\end{remark}

% Under this more realistic assumption, the use of the lazy Hessian approach becomes a reasonable choice, which is also the main motivation for this work.

% By this proposition, finding a saddle point of the function $f$ is equivalent to a \textit{variational inequality} problem~\citep{hartman1966some,facchinei2003finite}, which corresponds to finding a solution $\vz^*$ such that
% \begin{align*}
%     \langle \vg(\vz), 
% \end{align*}

\section{Algorithms and convergence analysis} \label{sec:alg}

In this section, we present novel second-order methods for solving minimax optimization problems~(\ref{prob:min-max}). We present LEN and its convergence analysis for convex-concave minimax problems in Section~\ref{sec:LEN}. We generalize LEN for strongly-convex-strongly-concave minimax problems by presenting LEN-restart in Section~\ref{sec:LEN-restart}. 
We discuss the details of solving 
minimax cubic-regularized sub-problem, present detailed implementation of LEN, and give the total computational complexity of proposed methods in Section~\ref{sec:imple}.

\subsection{The LEN algorithm for convex-concave problems}
\label{sec:LEN}

We propose LEN for convex-concave problems in Algorithm~\ref{alg:LEN}.
Our method builds on 
the optimal Newton 
Proximal Extragradient (NPE) method 
~\citep{monteiro2012iteration,bullins2022higher,adil2022optimal,lin2022explicit}. The only change is that we reuse the Hessian from previous iterates, as colored {\color{blue} in blue}. Each iteration of LEN contains the following two steps:
\begin{align}
\label{eq:LEN-update}
\begin{cases}
\mF(\vz_t)+\nabla \mF({\color{blue}\vz_{\pi(t)}}) (\vz_{t+1/2}-\vz_t) + M\|{\vz}_{t+1/2}-\vz_{t}\|({\vz}_{t+1/2}-\vz_t)={\bf 0}, ~~ & \text{(Implicit Step)}\\ \addlinespace
    \vz_{t+1} = \vz_t-\dfrac{\mF({\vz}_{t+1/2})}{M\|{\vz}_{t+1/2}-\vz_t\|}. ~~ &\text{(Explicit Step)}
\end{cases}
\end{align}
The first step (implicit step) solves a cubic regularized sub-problem based on the $\nabla \mF(\vz_{\pi(t)})$ computed at the latest snapshot point and $\mF(\vz_t)$ at the current iteration point. This step is often viewed as an oracle~\citep{bullins2022higher,adil2022optimal,lin2022explicit} as there exists efficient solvers, which will also be discussed in Section \ref{sec:imple}.
The second one (explicit step) 
conducts an
extra gradient step based on $\mF(\vz_{t+1/2})$. 
%Since the Hessian oracle is typically much more expensive than the gradient oracle, reusing the Hessian makes LEN more computation-efficient than the existing optimal second-order methods.

Reusing the Hessian in the implicit step makes each iteration much cheaper, but would
cause additional errors compared to previous methods~\citep{monteiro2012iteration,huang2022approximation,adil2022optimal,lin2022explicit}.
The error resulting from the lazy Hessian updates is formally
characterized by the following theorem.
%To see whether the total complexity could be improved using lazy Hessian updates, we need to analyze the progress of each iteration in the algorithm, formally described in the following lemma.
\begin{lem} \label{lem:LEN}
Suppose that Assumption \ref{asm:prob-lip-hes} and \ref{asm:prob-cc} hold.  For any $\vz \in \sR^d$, Algorithm \ref{alg:LEN} ensures
\begin{align*}
    \quad \gamma_t^{-1} \langle  \mF(\vz_{t+1/2}), \vz_{t+1/2} - \vz \rangle &\le \frac{1}{2} \Vert \vz_{t} - \vz \Vert^2 - \frac{1}{2} \Vert \vz_{t+1} - \vz \Vert^2  -\frac{1}{2} \Vert \vz_{t+1/2} - \vz_{t+1} \Vert^2 \\
    &\quad - \frac{1}{2} \Vert \vz_t - \vz_{t+1/2} \Vert^2  + \frac{\rho^2}{2 M^2} \Vert \vz_t - \vz_{t+1/2} \Vert^2 + \underbrace{\frac{2 \rho^2}{M^2} \Vert \vz_{\pi(t)} - \vz_t \Vert^2}_{(*)}.
\end{align*}
\end{lem}

Above, (*) is the error from lazy Hessian updates. Note that
(*) vanishes when
the current Hessian is used. For lazy Hessian updates, the error would accumulate in the epoch. 

The key step in our analysis shows that we can use the negative terms in the right-hand side of the inequality in Lemma \ref{lem:LEN} to bound the accumulated error by choosing $M$ sufficiently large, with the help of the following technical lemma.
\begin{lem} \label{lem:seq}
    For any sequence of positive numbers $\{ r_t\}_{t \ge 0}$, it holds for any $m \ge 2$ that 
$ \sum_{t=1}^{m-1} \left( \sum_{i=0}^{t-1} r_i \right)^2 \le \frac{m^2}{2} \sum_{t=0}^{m-1} r_t^2.$
% \begin{align*}
%     \sum_{t=1}^{m-1} \left( \sum_{i=0}^{t-1} r_i \right)^2 \le \frac{m^2}{2} \sum_{t=0}^{m-1} r_t^2.
% \end{align*}
\end{lem}
% We combine the above two lemmas to 
% show the convergence of the proposed LEN with any $m \ge 1$.

When $m=1$, the algorithm reduces to the NPE algorithm~\citep{monteiro2012iteration,bullins2022higher,adil2022optimal,lin2022explicit} without using lazy Hessian updates. When $m \ge 2$, we use Lemma \ref{lem:seq} to upper bound the error that arises from lazy Hessian updates.
Finally, we prove the following guarantee for our proposed algorithm.

% At each iteration, one Hessian oracle and one gradient oracle are queried. However, since Hessian oracles are typically more expensive than gradient oracles, we propose reusing an old Hessian may improve the total complexity.
% Specifically, we replace $\vH(\vz_t)$ in \Eqref{eq:CRN} with a snaphot Hessian $\vH(\vw_t)$ which updates per $m$ iterations. The resulting algorithm is presented in Algorithm \ref{alg:LEN}.

% If we suppose that one Hessian oracle is as expensive as $c$ gradient oracles, we should let $m =c$ to make the CRN step and EG step equivalently expensive.

\begin{algorithm*}[t]  
\caption{LEN$(\vz_0, T, m,M)$}  \label{alg:LEN}
\begin{algorithmic}[1] 
\STATE \textbf{for} $t=0,\cdots,T-1$ \textbf{do} \\
% \STATE \quad \textbf{if} $t \mod m =0$ \textbf{do} \\
% \STATE \quad \quad Update the snapshot $\vw_t = \vz_t$  and compute its Hessian $\nabla 
%  \mF(\vw_t) $ \\
% \STATE \quad \textbf{end if} \\
\STATE \quad Compute lazy cubic step, \textit{i.e.} find $\vz_{t+1/2}$ that satisfies
\begin{align*}
    \mF(\vz_t) = (\nabla \mF({\color{blue}\vz_{\pi(t)}}) + M \Vert \vz_t - \vz_{t+1/2} \Vert \mI_{d})  (\vz_t - \vz_{t+1/2}). 
\end{align*}
\\
\STATE \quad Compute $ \gamma_t =M\|\vz_t-\vz_{t+1/2}\|$. \\
\STATE \quad Compute extra-gradient step $ \vz_{t+1} = \vz_t -  \gamma_t^{-1} \mF(\vz_{t+1/2}) $. \label{line:extra}
\STATE \textbf{end for} \\
\STATE \textbf{return} $ \bar \vz_T = \frac{1}{\sum_{t=0}^{T-1} \gamma_t^{-1}} \sum_{t=0}^{T-1} \gamma_t^{-1} \vz_{t+1/2}$.
\end{algorithmic}
\end{algorithm*}

\begin{thm}[C-C setting] \label{thm:LEN}
    Suppose that Assumption \ref{asm:prob-lip-hes} and \ref{asm:prob-cc} hold. Let $\vz^* = (\vx^*,\vy^*)$ be a saddle point and $\beta = \Vert \vz_0 - \vz^* \Vert$. Set $M \ge 3 \rho m$. The sequence of iterates generated by Algorithm \ref{alg:LEN} is bounded $\vz_t \in \sB_{\beta}(\vz^*), ~ \vz_{t+1/2} \in \sB_{3\beta}(\vz^*) , \quad \forall t = 0,\cdots,T-1,$
    % \begin{align*}
    %     \vz_t, ~ \vz_{t+1/2} \in \sB_{\beta}(\vz^*) , \quad \forall t = 0,\cdots,T-1,
    % \end{align*}
    and satisfies the following ergodic convergence:
    \begin{align*}
        {\rm Gap}(\bar \vx_T, \bar \vy_T; 3\beta) \le \frac{32 M \Vert \vz_0 - \vz^* \Vert^3}{T^{3/2}}.
    \end{align*}
Let $M = 3 \rho m $. Algorithm \ref{alg:LEN} finds an $\epsilon$-saddle point within $\gO(m^{2/3} \epsilon^{-2/3} )$ iterations.
% Under Assumption \ref{asm:arith-cmp}, Algorithm \ref{alg:LEN} with $m = \Theta(d) $ takes the total complexity of $\gO( N d + N d^{2/3} \epsilon^{-2/3})$ to call the oracles.
\end{thm}
\paragraph{Discussion on the computational complexity of the oracles.}
Theorem~\ref{thm:LEN} indicates that LEN requires $\gO(m^{2/3} \epsilon^{-2/3} )$ calls to $\mF(\cdot)$ and $\gO(m^{2/3}\epsilon^{-2/3}/m+1)$ calls to $\nabla \mF(\cdot)$ to find the $\epsilon$-saddle point. Under Assumption~\ref{asm:arith-cmp}, the computational cost to call the oracles $\mF(\cdot)$ and $\nabla \mF(\cdot)$ is
\begin{align}
\label{eq:compu-cost}
 \text{Oracle Computational Cost} = \gO\left(N\cdot m^{2/3}\epsilon^{-2/3} + (Nd)\cdot\left({\epsilon^{-2/3}}{m^{-1/3}}+1\right)\right).
\end{align}
Taking $m=\Theta(d)$ minimizes (\ref{eq:compu-cost}) to $\gO(Nd+N\textcolor{blue}{d^{2/3}}\epsilon^{-2/3})$.
Compared to state-of-the-art second-order methods~\citep{monteiro2012iteration,bullins2022higher,adil2022optimal,lin2022explicit}, whose computational cost in terms of the oracles is $\gO(Nd\epsilon^{-2/3})$ since they require to query $\nabla \mF(\cdot)$ at each iteration, our methods significantly improve their results by a factor of $d^{1/3}$.

% \begin{remark}
% When $m =1$, the convergence rate of LEN recovers the optimal second-order methods~\citep{huang2022approximation,lin2022explicit,adil2022optimal} when simultaneously querying $\mF$ and  $\nabla \mF$. However, when $m=1$ the corresponding complexity of $\gO(  N d  \epsilon^{-2/3})$ to call the oracles is suboptimal compared to the $\gO(Nd + N d^{2/3} \epsilon^{-2/3})$ complexity we establish in Theorem \ref{thm:LEN} when $m = \Theta(d)$. 
% \end{remark}

It is worth noticing that the computational cost of an algorithm includes both the computational cost of calling the oracles, which we have discussed above, and the computational cost of performing the updates (\textit{i.e.} solving auxiliary problems) after accessing the required oracles.
We will give an efficient implementation of LEN and analyze the total computational cost later in Section~\ref{sec:imple}.

% Up to now, we have considered the computational complexity of calling oracles under Assumption~\ref{asm:arith-cmp}. 
% The overall complexity of the algorithm should also include the complexity arising from solving auxiliary problems during the iterations of the algorithm. 

% This will be addressed when discussing the implementation of the algorithm in Section \ref{sec:imple}.

\subsection{The LEN-restart algorithm for strongly-convex-strongly-concave problems}
\label{sec:LEN-restart}
We generalize LEN to solve the strongly-convex-strongly-concave minimax optimization by incorporating the restart strategy introduced by~\citet{huang2022approximation,lin2022perseus}. We propose the LEN-restart in Algorithm \ref{alg:LEN-restart}, which works in epochs. 
Each epoch of LEN-restart invokes LEN (Algorithm~\ref{alg:LEN}), which gets $\vz^{(s)}$ as inputs and outputs $\vz^{(s+1)}$.

The following theorem shows that the sequence $\{ \vz^{(s)}\} $ enjoys a superlinear convergence in epochs. 
Furthermore, the required number of iterations in each epoch to achieve such a superlinear rate
is only a constant.

\begin{algorithm*}[t]  
\caption{LEN-restart$(\vz_0, T, m,M, S)$}  \label{alg:LEN-restart}
\begin{algorithmic}[1] 
\STATE $\vz^{(0)} = \vz_0$ \\
\STATE \textbf{for} $s=0,\cdots,S-1$\\
\STATE \quad $\vz^{(s+1)} =\text{LEN}(\vz^{(s)},T,m,M) $ \\
\textbf{end for}
\end{algorithmic}
\end{algorithm*}

\begin{thm}[SC-SC setting] \label{thm:restart-LEN}
Suppose that Assumptions \ref{asm:prob-lip-hes} and \ref{asm:prob-scsc} hold. Let $\vz^* = (\vx^*,\vy^*)$ be the unique saddle point. Set $M = 3 \rho m $ as Theorem \ref{thm:LEN} and $T = \left( \frac{2 M \Vert \vz_0 - \vz^* \Vert}{\mu} \right)^{2/3}.$
% \begin{align*}
%     T = \left( \frac{2 M \Vert \vz_0 - \vz^* \Vert}{\mu} \right)^{2/3}.
% \end{align*}
Then the sequence of iterates generated by Algorithm \ref{alg:LEN-restart} converge to $\vz^*$ superlinearly as $ \Vert \vz^{(s)} - \vz^* \Vert^2 \le \left( \frac{1}{2}\right)^{(3/2)^{s}} \Vert \vz_0 - \vz^* \Vert^2.$ 
% :
% \begin{align*}
%     \Vert \vz^{(s)} - \vz^* \Vert^2 \le \left( \frac{1}{2}\right)^{(3/2)^{s}} \Vert \vz_0 - \vz^* \Vert^2.
% \end{align*}
In other words, Algorithm \ref{alg:LEN-restart} finds a point $\vz^{(s)}$ such that $\Vert \vz^{(s)} - \vz^* \Vert \le \epsilon$ within $ S = \log_{3/2} \log_{2} (\nicefrac{1}{\epsilon}) $  epochs.
The total number of inner loop iterations is given by
\begin{align*}
    TS = \gO \left(m^{2/3} \kappa^{2/3} \log \log (\nicefrac{1}{\epsilon}) \right).
\end{align*}
% where $\kappa: = \rho / \mu$ is the condition number. 
Under Assumption \ref{asm:arith-cmp}, Algorithm \ref{alg:LEN-restart} with $m = \Theta(d)$ takes the computational complexity of $\gO( (N d+  N d^{2/3} \kappa^{2/3}) \log \log (\nicefrac{1}{\epsilon}))$ to call the oracles $\mF(\,\cdot\,)$ and $\nabla \mF(\,\cdot\,)$.
\end{thm}

% \begin{remark}
% When $m =1$, The convergence rate of {\rm LEN-restart} recovers the global superlinear rate of $\gO(\kappa^{2/3} \log \log (\nicefrac{1}{\epsilon}))$ for ARE-restart method. However, our result shows that $m=1$ is a suboptimal choice, and choosing $m=\Theta(d)$ can improve the current best-known results by a factor of $d^{1/3}$.
% \end{remark}

\subsection{Implementation Details and computational complexity Analysis} \label{sec:imple} 

We provide details of implementing the cubic regularized Newton oracle (Implicit Step, (\ref{eq:LEN-update})). Inspired by \citet{monteiro2012iteration,bullins2022higher,adil2022optimal,lin2022explicit}, we transform the sub-problem into a root-finding problem for a univariate function. 

% \begin{lem}
% Suppose Assumption \ref{asm:prob-lip-hes} and \ref{asm:prob-cc} hold for function $f:\sR^{d_x} \times \sR^{d_y} \rightarrow \sR$ and define
% \begin{align*}
%     \vz_{t+1/2}(\eta; \vz_t) := \vz_t - (\nabla \mF(\vz_{\pi(t)})) + \eta^{-1} \mI_d)^{-1} 
% \end{align*}
% % $\mF$ be its GDA field. 
% \end{lem}

% We provide details of implementing the cubic regularized Newton oracle (Implicit Step, (\ref{eq:LEN-update})). % we transform the sub-problem into a root-finding problem for a univariate function. 

\begin{lem}[Section 4.3 \citet{lin2022explicit}]
Suppose Assumption \ref{asm:prob-lip-hes} and \ref{asm:prob-cc} hold for function $f:\sR^{d_x} \times \sR^{d_y} \rightarrow \sR$ and let
$\mF$ be its GDA field. 
Define $\gamma_t = M \Vert \vz_{t+1/2} - \vz_t \Vert$.
The cubic regularized Newton oracle (Implicit Step, (\ref{eq:LEN-update})) can be rewritten as:
\begin{align*}
    \vz_{t+1/2} = \vz_t - ( \nabla \mF(\vz_{\pi(t)}) +  \gamma_t \mI_d )^{-1} \mF(\vz_t),
\end{align*}
which can be implemented by finding the root of the following univariate function:
\begin{align} \label{eq:phi}
    \phi(\gamma):= M \left \Vert \right 
    ( \nabla \mF(\vz_{\pi(t)}) +  \gamma \mI_d )^{-1} \mF(\vz_t)
    \Vert - \gamma.
\end{align}
Furthermore, the function $\phi(\gamma)$ is strictly decreasing 
when $\lambda>0$.
%, \textit{i.e.} it always holds that $\phi'(\lambda) <0$ when $\lambda>0$. 
\end{lem}

From the above lemma, to implement the cubic regularized Newton oracle, it suffices to find the root of a strictly decreasing function $\phi(\gamma)$, which can be solved within $\tilde \gO(1)$ iteration.
%by either the univariate Newton method~\citep[Appendix A.1]{nesterov2018lectures} or
% binary search on $\phi(\cdot)$~\citep[Section 5.1]{bullins2022higher}. 
The main operation is to solve the following linear system:
\begin{align} \label{eq:linear-sys}
    (\nabla \mF(\vz_{\pi(t)}) + \gamma \mI_d) \vh = \mF(\vz_t).
\end{align}
Naively solving this linear system at every iteration still results in an expensive computational complexity of $\gO(d^{3})$ per iteration. 

We present a computationally efficient way to implement LEN by leveraging the Schur factorization at the snapshot point $\nabla \mF(\vz_{\pi(t)}) = \mQ \mU \mQ^{-1}$,  where $\mQ \in \sC^{d \times d}$ is a unitary matrix and $\mU \in \sC^{d \times d}$ is an upper-triangular matrix. 
Then solving the linear system~(\ref{eq:linear-sys}) is equivalent to
\begin{align} \label{eq:tri-linear-sys}
    \vh = \mQ ( \mU + \gamma \mI_d)^{-1} \mQ^{-1} \mF(\vz_t).
\end{align}
The final implementable algorithm is presented in Algorithm \ref{alg:LEN-imple}.

Now, we are ready to analyze the total computational complexity of LEN, which can be divided into the following two parts:
{
\begin{align*}
     \text{Computational Cost}
    &= \text{Oracle Computational Cost} + \text{Update Computational Cost},
\end{align*}
}
\!where the first part has been discussed in Section~\ref{sec:LEN}. 
Regarding the update computational cost, 
the Schur decomposition with an computational complexity $\gO(d^{3})$ is required once every $m$ iterations. 
After Schur's decomposition has been given at the snapshot point, the dominant part of the update computational complexity is solving the upper-triangular linear system~(\ref{eq:tri-linear-sys}) with the back substitution algorithm within the computational complexity ${\gO}(d^2)$.
Thus, we have
\begin{align}
\label{eq:update-cost}
    \text{Update Computational Cost} = \tilde{\gO}\left(d^2\cdot m^{2/3}\epsilon^{-2/3}+d^{3}\cdot\left({\epsilon^{-2/3}}{m^{-1/3}}+1\right)\right),
\end{align}
and the total computational cost of LEN is
\begin{align}
\label{eq:total-compu-cost}
  \text{Computational Cost} \overset{(\ref{eq:compu-cost}),(\ref{eq:update-cost})}{=} \!\tilde{\gO}\left((d^2+N)\cdot m^{2/3}\epsilon^{-2/3} + (d^3+Nd)\cdot\left({\epsilon^{-2/3}}{m^{-1/3}}+1\right)\!\right).   
\end{align}
% The Schur decomposition at the snapshot point requires an $\gO(d^{3})$ computational complexity. Since the snapshot point is changed only per $m$ iterations, the computational complexity of Schur decomposition is only $\gO(d^{3} / m)$ on average. And (\ref{eq:tri-linear-sys}) can be computed in $\gO(d^2)$ computational complexity by solving the upper-triangular linear system with the back substitution algorithm. 
By taking $m=\Theta(d)$, we obtain the best computational complexity in (\ref{eq:total-compu-cost}) of LEN, which is formally stated in the following theorem.
\begin{thm}[C-C setting] \label{thm:LEN-CC-complexity}
Under the same setting of Theorem \ref{thm:LEN}, Algorithm \ref{alg:LEN-imple} with $m = \Theta(d)$ finds an $\epsilon$ -saddle point with computational complexity $\tilde \gO((N+d^2) (d+ d^{2/3} \epsilon^{-2/3})$.
\end{thm}

We also present the total computational complexity of LEN-restart for SC-SC setting.
\begin{cor}[SC-SC setting]
\label{thm:LEN-SCSC-complexity}
Under the same setting as in Theorem \ref{thm:restart-LEN}, Algorithm \ref{alg:LEN-restart} implemented in the same way as Algorithm \ref{alg:LEN-imple} with $m = \Theta(d)$ finds an $\epsilon$-root with computational complexity $\tilde \gO((N+d^2) (d+ d^{2/3} \kappa^{2/3})$.
\end{cor}

% In this case, 
% we obtain the following computational complexity  under Assumption \ref{asm:arith-cmp}:
% \begin{enumerate}
%     \item For the C-C case, Theorem \ref{thm:LEN} indicates that the computational complexity of LEN (Algorithm \ref{alg:LEN}) to find an $\epsilon$-saddle point is .
%     \item For the SC-SC case, Theorem \ref{thm:restart-LEN} indicates that the computational complexity of LEN-restart (Algorithm \ref{alg:LEN-restart}) to find an $\epsilon$-root is $\tilde \gO((N+d^2) (d+ d^{2/3} \kappa^{2/3})$.
% \end{enumerate}

In both cases, our proposed algorithms improve the total computational cost of the optimal second-order methods~\citep{monteiro2012iteration,bullins2022higher,adil2022optimal,lin2022explicit} by a factor of $d^{1/3}$.
% It is worth noting that, our result is the first to achieve a polynomial improvement in running time since  \citet{monteiro2012iteration}.

\begin{algorithm*}[t]  
\caption{Implementation of LEN $(\vz_0, T, m,M)$}  \label{alg:LEN-imple}
\begin{algorithmic}[1] 
\STATE \textbf{for} $t=0,\cdots,T-1$ \textbf{do} \\
\STATE \quad \textbf{if} $t \mod m =0$ \textbf{do} \\
\STATE \quad \quad Compute the Schur decomposition such that $  \nabla \mF(\vz_t) = \mQ \mU \mQ^{-1}$. \\
\STATE \quad \textbf{end if} \\
% \STATE \quad $r = \Vert \vz_t - \vz_{t-1/2} \Vert$.
% \STATE \quad \textbf{repeat} \\
% \STATE \quad \quad Compute $\varphi(r) $ and $\varphi(r')$ according to Proposition \ref{prop:form-phi}. \\
\STATE \quad Let $\phi(\,\cdot\,)$ defined as \ref{eq:phi} and compute $\gamma_t$ as its root by a binary search.
% \STATE \quad \quad Update $r = r - \varphi(r) / \varphi(r)' $ \\
% \STATE \quad \textbf{until} $\vert \varphi(r) \vert \le \zeta$ \\
\STATE \quad Compute lazy cubic step $\vz_{t+1/2} = \mQ (\mU + \gamma_t \mI_d)^{-1} \mQ^{-1} \mF(\vz_t).$ \\
% \STATE \quad Compute $ \gamma_t = M\Vert  \vz_t - \vz_{t+1/2} \Vert$, $ \eta_t = 1/\gamma_t$. \\
\STATE \quad Compute extra-gradient step $ \vz_{t+1} = \vz_t - \gamma_t^{-1} \mF(\vz_{t+1/2}) $.
\STATE \textbf{end for} \\
\STATE \textbf{return} $ \bar \vz_T = \frac{1}{\sum_{t=0}^{T-1} \gamma_t^{-1}} \sum_{t=0}^{T-1} \gamma_t^{-1} \vz_{t+1/2}$.
\end{algorithmic}
\end{algorithm*}

\begin{remark}
In the main text, we assume the use of the classical algorithm for matrix inversion/decomposition, which has a computational complexity of $\gO(d^3)$. 
The fast matrix multiplication proposed by researchers in the field of theoretical computer science only requires a complexity of
$d^{\omega}$,
where the best known
$\omega$ is currently around $2.371552$~\citep{williams2024new}. This also implies faster standard linear algebra operators including Schur decomposition and matrix inversion~\citep{demmel2007fast}.
However, the large hidden constant factors in these fast matrix multiplication algorithms
mean that the matrix dimensions necessary for these algorithms to be superior to classical algorithms are much larger than what current computers can effectively handle.
Consequently, these algorithms are not always used in practice. 
We present the computational complexity of using fast matrix operations in Appendix \ref{apx:fast-matrix}.
\end{remark}

In Appendix \ref{apx:inexact}, we 
extend our algorithms to allow inexact auxiliary CRN sub-problem solving and analyze the total complexity. 
Specifically, we design an efficient sub-procedure (Algorithm \ref{alg:line-search-eta}) to solve the CRN sub-problem to desired accuracy in only $\gO(\log \log (1/\epsilon))$ number of linear system solving. 
It tightens the $\gO(\log (1/\epsilon))$ iteration complexity in \citep{bullins2022higher,adil2022optimal}. 
Additionally,  \citep{bullins2022higher,adil2022optimal}  assume ${\sigma_{\min}}( \nabla \mF(\vz)) \ge \mu$ for some positive constant $\mu$, which makes the problem similar to strongly-convex(-strongly-concave) problems, while our analysis does not require such an assumption.

\section{Numerical Experiments}
\label{sec:exp}
We conduct our algorithms on a regularized bilinear min-max problem and fairness-aware machine learning tasks. We include EG~\citep{korpelevich1976extragradient} and NPE~\citep{monteiro2012iteration,bullins2022higher,adil2022optimal,lin2022explicit} (which is our algorithm with $m=1$) as baselines, since they are the optimal first- and second-order methods for convex-concave minimax problems, respectively. 
We run the programs on an AMD EPYC 7H12 64-Core Processor. \footnote{The source codes are available at \url{https://github.com/TrueNobility303/LEN}.}
% and the source codes are available on github \footnote{\url{https://github.com/TrueNobility303/LEN}} for reproduction.

%We do not compare our methods with quasi-Newton methods~\cite{liu2022quasi} since they only work for SC-SC problems and do not converge globally for general objectives.

\subsection{Regularized bilinear min-max problem}

\begin{figure}[t]
    \centering
    \begin{tabular}{c c c}
     \includegraphics[scale=0.26]{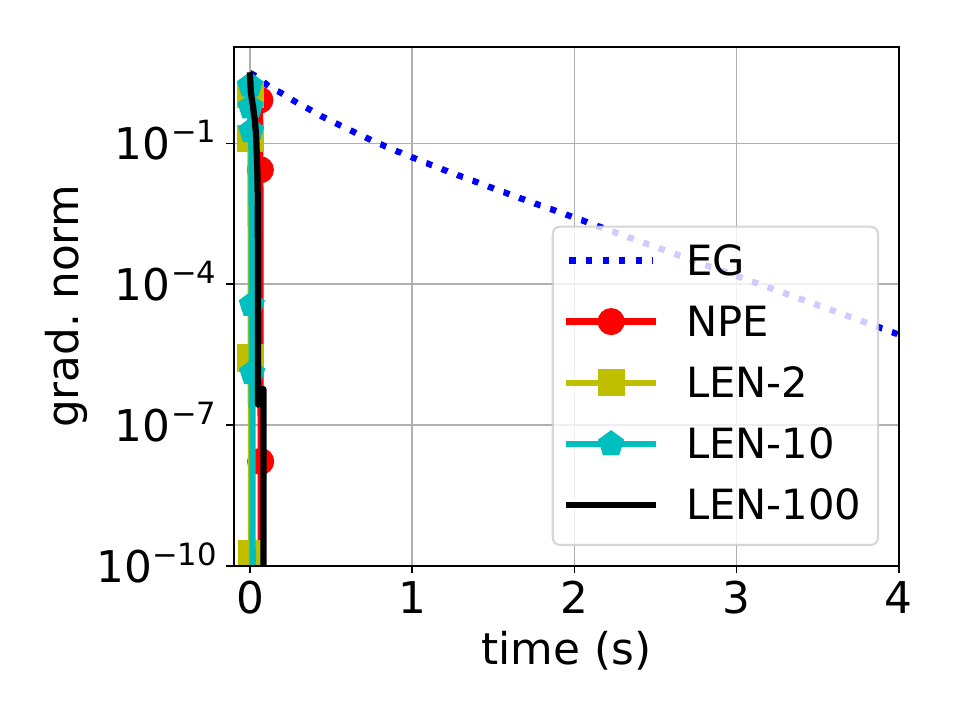}    & \includegraphics[scale=0.26]{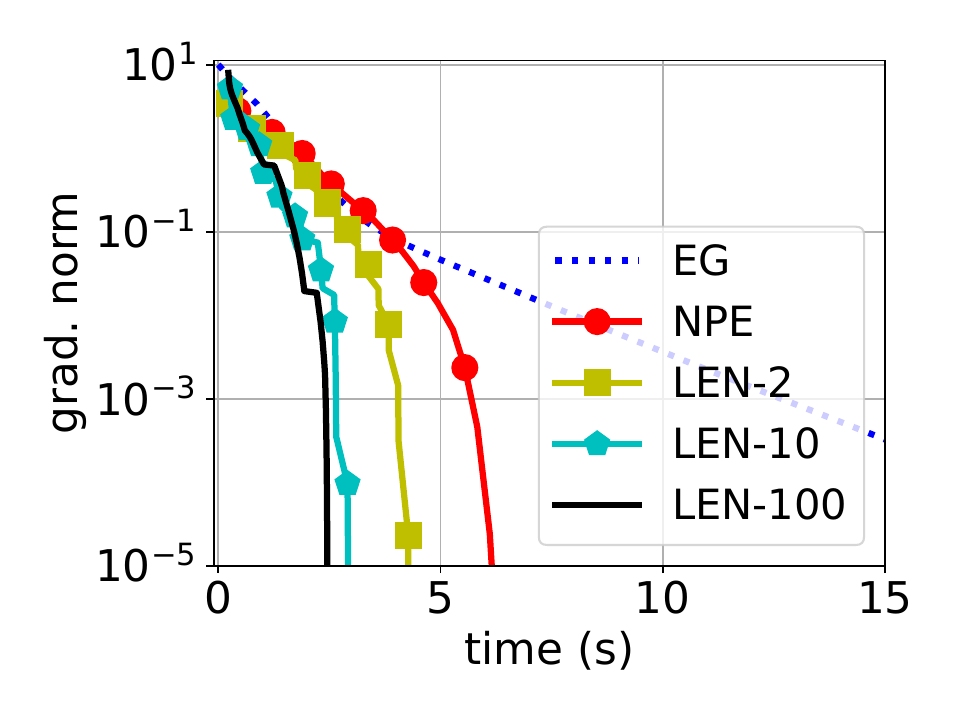} & \includegraphics[scale=0.26]{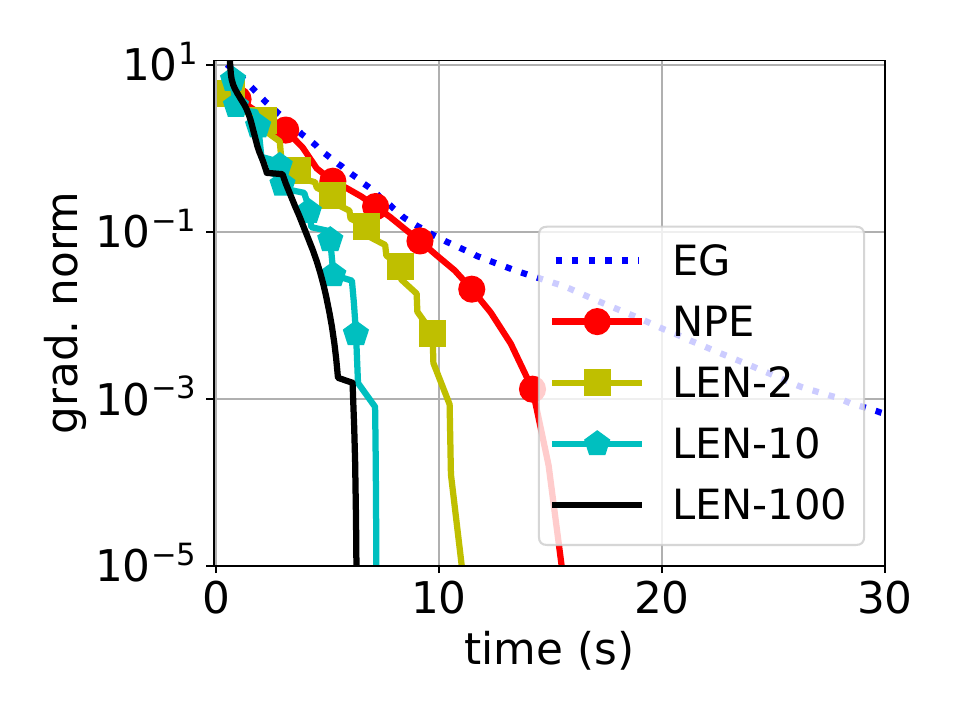}  \\
    (a) $n = 10$     &  (b) $n=100$ & (c) $n=200$ \\
    % \includegraphics[scale=0.26]{fig/Synthetic.10.dist.pdf}    & \includegraphics[scale=0.26]{fig/Synthetic.100.dist.pdf} & \includegraphics[scale=0.26]{newfig/Synthetic.200.dist.pdf}  \\
    % (d) $n = 10$     &  (e) $n=100$ & (f) $n=200$ \\
    \end{tabular}
    \caption{We demonstrate running time \textit{v.s.} gradient norm $\Vert \mF(\vz) \Vert$ for Problem (\ref{eq:cubic-toy}) with different sizes: $n \in \{ 10, 100,200\}$. }
    \label{fig:toy}
\end{figure}

We first conduct numerical experiments on the cubic regularized bilinear min-max problem considered in the literature \citep{lin2022explicit,jiang2024adaptive}:
{\small
\begin{align} \label{eq:cubic-toy}
    \min_{\vx \in \sR^n} \max_{\vy \in \sR^n} f(\vx,\vy) = \frac{\rho}{6} \Vert \vx \Vert^3  + \vy^\top (\mA \vx - \vb).
\end{align}
}
\!\!The function $f(\vx,\vy)$ is convex-concave and has $\rho$-Lipschitz continuous Hessians. 
The unique saddle point $\vz^*= (\vx^*, \vy^*)$ of $f(\vx,\vy)$ can be explicitly calculated as  $\vx^* = \mA^{-1} \vb$ and $\vy^* = - \rho \Vert \vx^* \Vert^2 (\mA^\top)^{-1} \vx^* / 2$, so we can compute the distance to $\vz^*$ to measure the performance of the algorithms.
Following \citet{lin2022explicit}, we generate each element in $\vb$ as independent Rademacher variables in $\{ -1, +1\}$, set $\rho = 1/(20n)$ and the matrix 
{\small
$
    \mA = 
    \begin{bmatrix}
        1 & -1 \\
        % & 1& - 1 \\
         & \ddots & \ddots \\
         & & 1 & -1 \\
         & & & 1 \\
        \end{bmatrix}.
$}

We compare our methods with the baselines on different sizes of the problem: $n \in \{10,100,200 \}$. For EG, we tune the stepsize in $ \{1, 0.1, 0.01, 0.001\}$. 
 For LEN, we vary $m$ in $\{1, 2, 10, 100\}$.
 %For a fair comparison, we solve the CRN subproblem in both ARE and LEN with the same sub-solver as we described previously in Section \ref{sec:imple}, which we found faster than the ``fsolve'' function in `scipy'' package. 
 The results of the running time against $\Vert \mF(\vz) \Vert$ are presented in Figure \ref{fig:toy}.

\subsection{Fairness-Aware Machine Learning}

We then examine our algorithm for the task of fairness-aware machine learning.
Let $ \{ \va_i, b_i,c_i\}_{i=1}^n$ be the training set, where $\va_i \in \sR^{d_x}$ denotes the features of the $i$-th sample, $b_i \in \{-1,+1\}$ is the corresponding label, and 
$c_i \in \{-1,+1 \}$ is an additional feature that is required to be protected and debiased. For example, $c_i$ can denote gender. \citet{zhang2018mitigating} proposed to solve the following minimax problem to mitigate unwanted bias of $c_i$ by adversarial learning:
\begin{align} \label{eq:fair}
    \min_{\vx \in \sR^{d_x}} \max_{y \in \sR} \frac{1}{n} \sum_{i=1}^n \ell(b_i \va_i^\top \vx) - \beta \ell(c_i y \va_i^\top \vx) + \lambda \Vert \vx \Vert^2 -\gamma y^2,
\end{align}
where $\ell$ is the logit function such that $\ell(t) = \log (1+\exp(-t))$.
We set $\lambda = \gamma =10^{-4}$ and $\beta = 0.5$. 
We conduct the experiments on datasets ``heart'' ($n=270$, $d_x = 13$)~\citep{chang2011libsvm},  ``adult'' ($n= 32,561$, $d_x = 123$)~\citep{chang2011libsvm} and ``law  school'' ($n=20,798$, $d_x= 380$)~\citep{le2022survey,liu2022partial}. For all the datasets, we choose ``gender'' as the protected feature.
For EG, we tune the stepsize in $\{0.1, 0.01, 0.001 \}$. For second-order methods (NPE and LEN), as we do not know the value of $\rho$ in advance, we view it as a hyperparameter and tune it in $\{1,10,100 \}$.
We set $m=10$ for LEN and we find that this simple choice  performs well in all the datasets we test.
We show the results of the running time against the gradient norm $\Vert \mF(\vz) \Vert$ in Figure \ref{fig:fairness}.

% For LEN, we set $m =  d_x +1$. For both ARE and LEN, we tune the parameter $M$ in $\{10^{1},10^{2},10^{3} \}$ and set $\zeta = 10^{-8}$. The univariate sub-solver usually terminates within less than $10$ iterations in this experiment. 

\begin{figure}[]
    \centering
    \begin{tabular}{c c c}
    \includegraphics[scale=0.26]{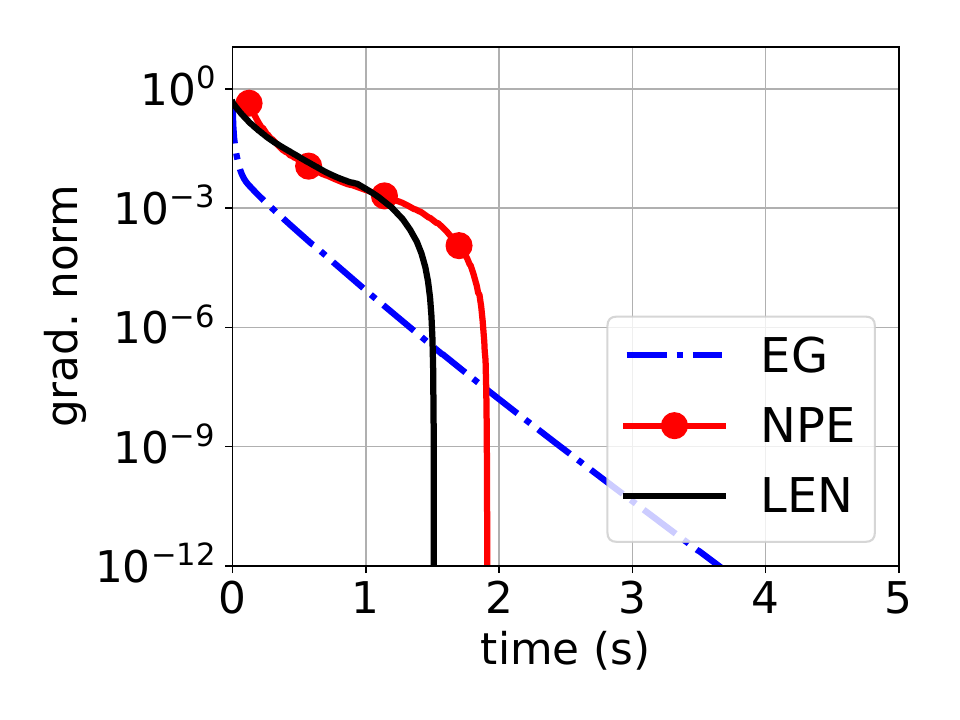} & 
    \includegraphics[scale=0.26]{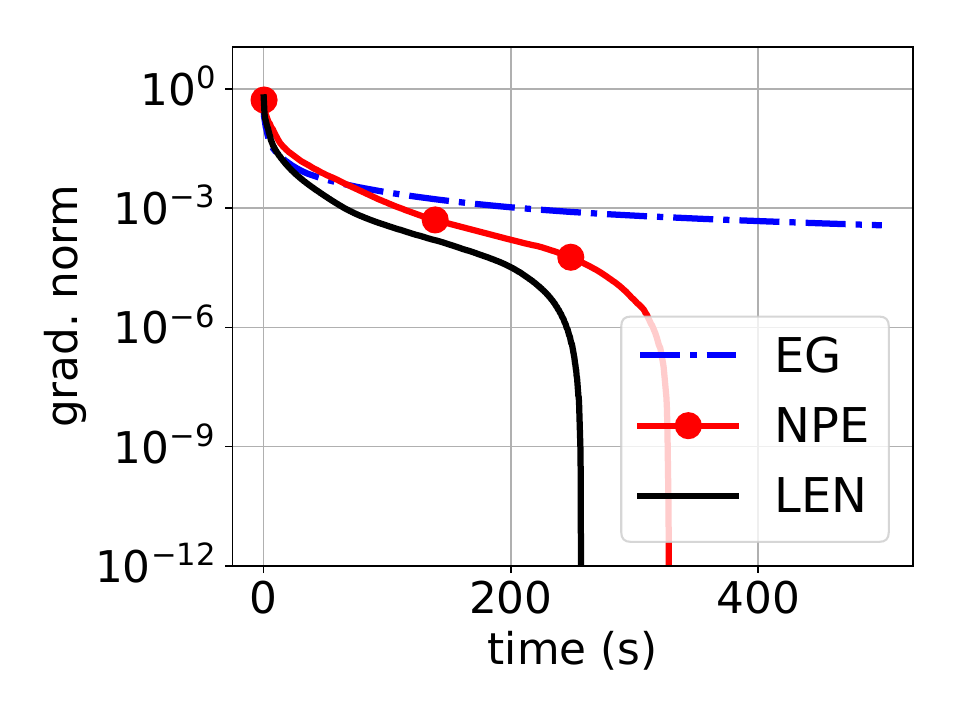} & 
    \includegraphics[scale=0.26]{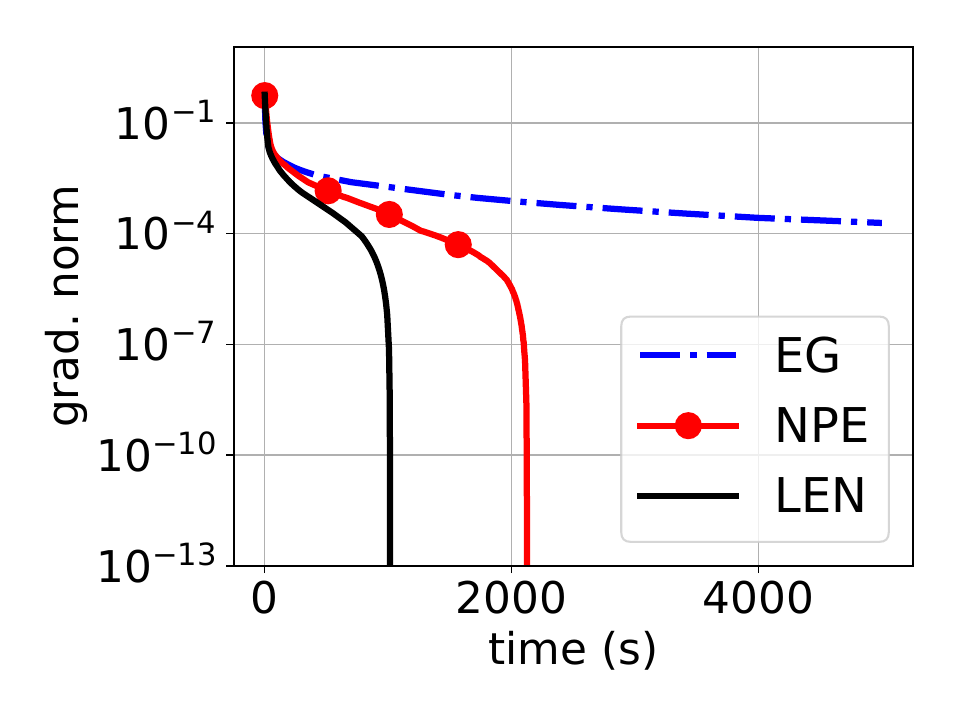} \\
    (a) heart & (b) adult & (c) law school 
    \end{tabular}
    \caption{We demonstrate running time \textit{v.s.} gradient norm $\Vert \mF(\vz) \Vert$ for fairness-aware machine learning task (Problem (\ref{eq:fair})) on datasets ``heart'', ``adult'', and ``law school''. }
    \label{fig:fairness}
\end{figure}

\section{Conclusion and future works}
\label{sec:conclu}
In this paper, we propose LEN and LEN-restart for C-C and SC-SC minimax problems, respectively. Using lazy Hessian updates, our methods improve the computational complexity of the current best-known second-order methods by a factor of $d^{1/3}$.  
Numerical experiments on both real and synthetic datasets also verify the efficiency of our method. 
% Our results show that the optimal second-order methods in terms of oracle complexity are not optimal in terms of the computation complexity.

% Discuss distributed implementation. xxx.

% Our results inspire us to rethink the \textit{optimality} in optimization.
% The existing second-order methods are called ``optimal'' when the computation of gradient and Hessian are regarded as the same cost. However, under the more realistic assumption that Hessian is more expensive than gradient, the original ``optimal'' methods become suboptimal. 

%This highlights the importance of defining a proper computation model.

% It is worth mentioning that our algorithm is highly distributed-friendly. In distributed optimization scenarios, to reduce the $\gO(d^2)$ communication complexity caused by directly communicating the Hessian matrix, one approach is to communicate only one Hessian-vector at a time, as done in \citet{liu2023communication}.
% After $d$ rounds, the server will obtain an old Hessian from $d$ rounds ago, which naturally results in a delayed Hessian. Our analysis can also be easily extended to the aforementioned distributed scenarios.

%Our work shows that the optimal methods under the oracle model may not be optimal in terms of computational complexity. This demonstrates the potential for designing more efficient algorithms when getting rid of the oracle model.

For future work, it will be interesting to extend our idea to adaptive second-order methods~\citep{wang2024fully,doikov2024super,carmon2022optimal,antonakopoulos2022extra,liu2022regularized} and stochastic problems with sub-sampled Newton methods~\citep{lin2022explicit,chayti2023unified,zhou2019stochastic,tripuraneni2018stochastic,wang2019stochastic}.
Besides, our methods only focus on the convex-concave case, it is also possible to reduce the Hessian oracle for nonconvex-(strongly)-concave problems~\citep{luo2022finding,lin2020near,yang2023accelerating,zhang2022sapd+,wang2024efficient} or study structured nonconvex-nonconcave problems
~\citep{zheng2024universal,diakonikolas2021efficient,yang2020global,lee2021fast,chen2022near}.

\section*{Acknowledgments}
Lesi Chen and Jingzhao Zhang are supported by the Shanghai Qi Zhi Institute Innovation Program. 
Chengchang Liu is supported by the National Natural Science Foundation of China (624B2125). 

\bibliography{iclr2025_conference}
\bibliographystyle{iclr2025_conference}

\newpage
\appendix

\section{Some Useful Lemmas}

\begin{lem} \label{lem:gap-gap}
    Recall the definition of restricted gap function in Definition \ref{dfn:gap}. For any point $(\hat \vx, \hat \vy)$ We have that
    \begin{align*}
        {\rm Gap}(\hat \vx, \hat \vy ; \beta ) \le \max_{\vz \in \sB_{\sqrt{2} \beta (\vz^*)}} \left\{ f(\hat \vx, \vy) - f(\vx, \hat \vy) \right\}, \quad \vz = (\vx,\vy).
    \end{align*}
\end{lem}

\begin{proof}
    By Definition \ref{dfn:gap}, we have that
    \begin{align*}
    {\rm Gap}(\hat \vx, \hat \vy; \beta) = \max_{\vx \in \sB_{\beta}(\vx^*), \vy \in \sB_{\beta}(\vy^*)} \left\{ f(\hat \vx, \vy ) -  f(\vx, \hat \vy) \right\} \le \max_{\vz \in \sB_{\sqrt{2} \beta}(\vz^*)} \left\{ f(\hat \vx, \vy ) -  f(\vx, \hat \vy) \right\}.
    \end{align*}
\end{proof}

\begin{lem}[Proposition 2.8 \citet{lin2022explicit}] \label{lem:avg}
Let
    \begin{align*}
        \bar \vx_t = \frac{1}{\sum_{i=0}^{t} \eta_i} \sum_{i=0}^{t} \eta_i \vx_i, \quad \bar \vy_t =  \frac{1}{\sum_{i=0}^{t} \eta_i} \sum_{i=0}^{t} \eta_i \vy_i.
    \end{align*}
Then under Assumption \ref{asm:prob-cc}, for any $\vz = (\vx,\vy)$, it holds that
\begin{align*}
    f(\bar \vx_t, \vy) - f(\vx , \bar \vy_t) \le \frac{1}{\sum_{i=0}^{t} \eta_i}  \sum_{i=0}^{t} \eta_i \langle \mF(\vz_i), \vz_i - \vz   \rangle. 
\end{align*}
\end{lem}

\section{Proof of Lemma \ref{lem:seq}}

\begin{proof}
    We prove the result by induction. 
    
    Apparently, it is true for $m=2$, which is the induction base. 
    
Assume that it holds for $m \ge 2$. Then
\begin{align*}
    &\quad \sum_{t=1}^m \left(  \sum_{i=0}^{t-1} r_i\right)^2  \\
    &= \sum_{t=1}^{m-1} \left(  \sum_{i=0}^{t-1} r_i\right)^2 + \left( \sum_{i=0}^{m-1} r_i \right)^2 \\
    &\le \frac{m^2}{2} \sum_{t=0}^{m-1} r_t^2 + m \sum_{t=0}^{m-1} r_t^2 \\
    &\le \left( \frac{m^2 + 2m}{2} \right) \sum_{t=0}^{m-1} r_t^2 \\
    &\le \frac{(m+1)^2}{2} \sum_{t=0}^{m-1} r_t^2.
\end{align*}
\end{proof}

\section{Proof of Lemma \ref{lem:LEN}}

\begin{proof}
{   Instead of directly providing a proof for Algorithm \ref{alg:LEN}, we give the proof for the more general inexact algorithm (Algorithm \ref{alg:LEN-inexact}), which recovers Algorithm \ref{alg:LEN} if $\alpha = 1$.}

For convenience, we denote $\eta_t = 1/\gamma_t$.
% %let $\pi(t) = t- t \mod m $ such that $\vz_{\pi(t)}$ is the latest snapshot $\vw_t$. We also denote
% \begin{align}
%     \gamma_t \triangleq M\|\vz_{t+1/2}-\vz_t\|~~~\text{and}~~~\eta_t \triangleq \frac{1}{\gamma_t}.
% \end{align}

For ant $\vz \in \sR^{d}$,
     we have
    \begin{align} \label{eq:EG}
    \begin{split}
&\quad \eta_t \langle  \mF(\vz_{t+1/2}), \vz_{t+1/2} - \vz \rangle \\
& =  \langle \vz_t - \vz_{t+1}, \vz_{t+1/2} - \vz \rangle \\
&= \langle \vz_t - \vz_{t+1},  \vz_{t+1} - \vz \rangle +  \langle \vz_t - \vz_{t+1}, \vz_{t+1/2} - \vz_{t+1} \rangle \\
&=\langle \vz_t - \vz_{t+1}, \vz_{t+1} - \vz \rangle +  \langle \vz_t - \vz_{t+1/2}, \vz_{t+1/2} - \vz_{t+1} \rangle + \langle \vz_{t+1/2} - \vz_{t+1}, \vz_{t+1/2} - \vz_{t+1} \rangle \\
&= \frac{1}{2} \Vert \vz_{t} - \vz \Vert^2 - \frac{1}{2} \Vert \vz_{t+1} - \vz \Vert^2 - \bcancel{\frac{1}{2} \Vert \vz_t - \vz_{t+1} \Vert^2} \\
&\quad + \bcancel{\frac{1}{2} \Vert \vz_t - \vz_{t+1} \Vert^2}  -\frac{1}{2} \Vert \vz_{t+1/2} - \vz_{t+1} \Vert^2 - \frac{1}{2} \Vert \vz_t - \vz_{t+1/2} \Vert^2 + \Vert \vz_{t+1/2} - \vz_{t+1} \Vert^2.  
 \end{split}
    \end{align}
Note that by the updates of the algorithm, we have that
\begin{align*}
    \gamma_t (\vz_t - \vz_{t+1/2}) &= \mF(\vz_t) + \nabla \mF(\vz_{\pi(t)}) (\vz_{t+1/2} - \vz_t), \\
    \gamma_t(\vz_t - \vz_{t+1}) &= \mF(\vz_{t+1/2}).
\end{align*}
It implies that
\begin{align}
\begin{split} \label{eq:tt}
    &\quad \vz_{t+1/2} - \vz_{t+1} \\
    &= \eta_t (\mF(\vz_{t+1/2})  - \mF(\vz_t)- \nabla \mF(\vz_{\pi(t)}) (\vz_{t+1/2} - \vz_t))) \\
    &= \eta_t( \mF(\vz_{t+1/2})  - \mF(\vz_t)- \nabla \mF(\vz_{t}) (\vz_{t+1/2} - \vz_t)) + \eta_t (\nabla \mF(\vz_{\pi(t)})) - \nabla \mF(\vz_t)) (\vz_{t+1/2} - \vz_t)
\end{split}
\end{align}
Note that $\nabla \mF$ is $\rho$-Lipschitz continuous.
Taking norm on both sides of (\ref{eq:tt}), we have that
\begin{align*}
    \Vert \vz_{t+1/2} - \vz_{t+1} \Vert &\le \frac{\rho \eta_t}{2}  \Vert \vz_{t+1/2} - \vz_t \Vert^2 + \rho \eta_t \Vert \vz_{\pi(t)} - \vz_t \Vert \Vert \vz_{t+1/2} - \vz_t \Vert \\
    &\le \frac{\rho}{2 M} \Vert \vz_{t+1/2} - \vz_t \Vert + \frac{\rho}{M}  \Vert \vz_{\pi(t)} - \vz_{t} \Vert,
\end{align*}
{  where we use the condition $M\|\vz_t-\vz_{t+1/2}\| \le \gamma_t $ in the last step.}

By Young's inequality, this further means
\begin{align*}
     \Vert \vz_{t+1/2} - \vz_{t+1} \Vert^2 \leq \frac{\rho^2}{2M^2}\Vert \vz_{t+1/2} - \vz_t \Vert^2 + \frac{2\rho^2}{M^2}  \Vert \vz_{\pi(t)} - \vz_{t} \Vert^2. 
\end{align*}
Plug the above inequality into the last term in (\ref{eq:EG}).
\begin{align*}
    &\quad \eta_t \langle  \mF(\vz_{t+1/2}), \vz_{t+1/2} - \vz \rangle \\
    &\le \frac{1}{2} \Vert \vz_{t} - \vz \Vert^2 - \frac{1}{2} \Vert \vz_{t+1} - \vz \Vert^2  -\frac{1}{2} \Vert \vz_{t+1/2} - \vz_{t+1} \Vert^2 \\
    &\quad - \frac{1}{2} \Vert \vz_t - \vz_{t+1/2} \Vert^2  + \frac{\rho^2}{2 M^2} \Vert \vz_t - \vz_{t+1/2} \Vert^2 + \frac{2 \rho^2}{M^2} \Vert \vz_{\pi(t)} - \vz_t \Vert^2.
\end{align*}
\end{proof}

\section{Proof of Theorem \ref{thm:LEN}} \label{apx:proof-main}

\begin{proof}

When $m=1$, the algorithm reduces to the NPE algorithm~\citep{monteiro2012iteration,bullins2022higher,adil2022optimal,lin2022explicit}. When $m \ge 2$, we use Lemma \ref{lem:seq} to bound the error that arises from lazy Hessian updates.

{Instead of directly providing a proof for Algorithm \ref{alg:LEN}, we give the proof for the more general inexact algorithm (Algorithm \ref{alg:LEN-inexact}), which recovers Algorithm \ref{alg:LEN} if $\alpha = 1$.}

Define $r_t = \Vert \vz_{t+1} - \vz_t \Vert$. 
By triangle inequality and Young's inequality, we have
\begin{align*}
    &\quad \eta_t \langle  \mF(\vz_{t+1/2}), \vz_{t+1/2} - \vz \rangle \\ 
    &\le  \frac{1}{2} \Vert \vz_{t} - \vz \Vert^2 - \frac{1}{2} \Vert \vz_{t+1} - \vz \Vert^2 - \left( \frac{1}{4} - \frac{\rho^2}{2 M^2} \right) \Vert \vz_t - \vz_{t+1/2} \Vert^2 \\
    &\quad  -  \left( \frac{1}{8} r_t^2 - \frac{2\rho^2}{M^2} \left( \sum_{i=\pi(t)}^{t-1} r_i \right)^2 \right).
\end{align*}
For any $1 \le s  \le m$.
Telescoping over $t = \pi(k) ,\cdots, \pi(k) + s - 1$, we have
\begin{align*}
    &\quad \sum_{t= \pi(k)}^{\pi(k)  +s -1} \eta_t \langle  \mF(\vz_{t+1/2}), \vz_{t+1/2} - \vz \rangle \\  
    &\le \frac{1}{2} \Vert \vz_{\pi(k)} - \vz \Vert^2 - \frac{1}{2} \Vert \vz_{\pi(k)+s}  - \vz \Vert^2 -  \left( \frac{1}{4} - \frac{\rho^2}{2 M^2} \right) \sum_{t= \pi(k)}^{\pi(k)  +s -1}  \Vert \vz_t - \vz_{t+1/2} \Vert^2 \\
    &\quad - \left( \frac{1}{8} \sum_{t=\pi(k)}^{\pi(k) + s -1} r_t^2 -  \frac{2 \rho^2}{M^2}  \sum_{t=\pi(k)+1}^{\pi(k) + s -1} \left(\sum_{i=\pi(k)}^{t-1} r_i \right)^2 \right).
\end{align*}
Applying Lemma \ref{lem:seq}, we further have
\begin{align*}
    &\quad \sum_{t= \pi(k)}^{\pi(k)  +s -1} \eta_t \langle  \mF(\vz_{t+1/2}), \vz_{t+1/2} - \vz \rangle \\   
    &\le \frac{1}{2} \Vert \vz_{\pi(k)} - \vz \Vert^2 - \frac{1}{2} \Vert \vz_{\pi(k)+s}  - \vz \Vert^2 -  \left( \frac{1}{4} - \frac{\rho^2}{2 M^2} \right) \sum_{t= \pi(k)}^{\pi(k)  +s -1}  \Vert \vz_t - \vz_{t+1/2} \Vert^2 \\
    &\quad - \left( \frac{1}{8} - \frac{\rho^2 s^2}{M^2}  \right) \sum_{t=\pi(k)}^{\pi(k) + s -1} r_t^2.
\end{align*}
Note that $s \le m $. Let $M \ge 3 \rho m$. Then
\begin{align*}
    &\quad \sum_{t= \pi(k)}^{\pi(k)  +s -1} \eta_t \langle  \mF(\vz_{t+1/2}), \vz_{t+1/2} - \vz \rangle \\
    &\le \frac{1}{2} \Vert \vz_{\pi(k)} - \vz \Vert^2 - \frac{1}{2} \Vert \vz_{\pi(k)+s}  - \vz \Vert^2 -  \frac{1}{8 }  \sum_{t= \pi(k)}^{\pi(k)  +s -1}  \Vert \vz_t - \vz_{t+1/2} \Vert^2. 
\end{align*}
Let $s = m$ and
further telescope over $k = 0, m , 2m, \cdots$. Then
\begin{align} \label{eq:tele}
    \sum_{t= 0}^{T} \eta_t \langle  \mF(\vz_{t+1/2}), \vz_{t+1/2} - \vz \rangle \le \frac{1}{2} \Vert \vz_0 - \vz \Vert^2 - \frac{1}{2} \Vert \vz_{T} - \vz \Vert^2 - \frac{1}{8} \sum_{t=0}^T \Vert \vz_t - \vz_{t+1/2} \Vert^2.
\end{align}
This inequality is the key to the convergence. It implies the following results. First, letting $\vz = \vz^*$ and using the fact that $ \langle \mF(\vz_{t+1/2}), \vz_{t+1/2} - \vz^* \rangle \ge 0$ according to monotonicity of $\mF$, we can prove the iterate is bounded
\begin{align} \label{eq:iterate-bounded}
    \Vert \vz_t - \vz^* \Vert \le \Vert \vz_0 - \vz^* \Vert, \quad {\rm and} \quad \Vert \vz_t - \vz_{t+1/2} \Vert \le 2 \Vert \vz_0 - \vz^* \Vert, \quad t =0,\cdots,T-1.
\end{align}
Then using triangle inequality, we obtain
\begin{align*}
    \Vert \vz_{t+1/2} - \vz^* \Vert \le 3 \Vert \vz_0 - \vz^* \Vert, \quad \forall t = 0,\cdots,T-1.
\end{align*}
Second, beginning with Lemma \ref{lem:gap-gap} we have that % as (\ref{eq:tele}) holds for all $\vz \in \sB_{3\beta}(\vz^*)$, Lemma 
% \ref{lem:avg} indicates
\begin{align} \label{eq:gap}
\begin{split}
    {\rm Gap}(\bar \vx_T, \bar \vy_T; 3\beta)
    &\le \max_{\vz \in \sB_{3 \sqrt{2} \beta}(\vz^*)} \left\{ f(\bar \vx_T, \vy ) -  f(\vx, \bar \vy_T) \right\} \\
    &\overset{(a)}{\le}\max_{\vz \in \sB_{3 \sqrt{2} \beta}(\vz^*)} \left\{  \frac{1}{\sum_{t=0}^{T} \eta_t} \sum_{t= 0}^{T} \eta_t \langle  \mF(\vz_{t+1/2}), \vz_{t+1/2} - \vz \rangle \right\} \\
    &\overset{(b)}{\le} \frac{ \max_{\vz \in \sB_{3 \sqrt{2} \beta}(\vz^*)} \{ \Vert \vz_0 - \vz \Vert^2 \}}{2 \sum_{t=0}^{T-1} \eta_t} \overset{(c)}{\le} \frac{ 16 \Vert \vz_0 - \vz^* \Vert^2 }{\sum_{t=0}^{T-1} \eta_t},
\end{split}
\end{align}
where (a) uses Lemma \ref{lem:avg}, (b) uses (\ref{eq:tele}) and (c) uses the fact that $\vz \in \sB_{3 \sqrt{2} \beta}(\vz^*)$.
Third, we can also use (\ref{eq:tele}) to lower bound $\sum_{t=0}^{T-1} \eta_t$. (\ref{eq:tele}) with $\vz = \vz^*$ implies 
\begin{align*}
    \sum_{t=0}^T \gamma_t^2 \le 4 \alpha^2 M^2 \Vert \vz_0 - \vz^* \Vert^2,
\end{align*}
{  where we use the condition $\gamma_t \le \alpha M\|\vz_t-\vz_{t+1/2}\|  $ in the last step.} Then
by Holder's inequality,
\begin{align*}
    T=  \sum_{t=0}^{T-1} \left(\eta_t\right)^{2/3} \left(\gamma_t^2\right)^{1/3} \le \left( \sum_{t=0}^{T-1} \eta_t \right)^{2/3} \left( \sum_{t=0}^{T-1} \gamma_t^2 \right)^{1/3}. 
\end{align*}
Therefore,
\begin{align} \label{eq:lower-eta}
\sum_{t=0}^{T-1} \eta_t \ge \frac{T^{3/2}}{2 \alpha M \Vert \vz_0 - \vz^* \Vert}.
\end{align}
We plug in (\ref{eq:lower-eta}) to (\ref{eq:gap}) and obtain that
\begin{align*}
    {\rm Gap}(\bar \vx_T, \bar \vy_T; 3 \beta) \le \frac{32 \alpha M \Vert \vz_0  - \vz^* \Vert^3 }{T^{3/2}}.
\end{align*}
{  
The desired theorem is the case $\alpha=1$.}
\end{proof}

\section{Proof of Theorem \ref{thm:restart-LEN}}

\begin{proof}
%     Under the SC-SC setting, the operator $\mF$ is strongly monotone:
% \begin{align*}
%     \langle \mF(\vz) - \mF(\vz'), \vz - \vz' \rangle \ge \mu \Vert \vz - \vz' \Vert^2, \quad \forall \vz,\vz' \in \sR^d.
% \end{align*}
Using the strongly monotonicity of operator $\mF$ in (\ref{eq:tele}), we obtain that 
\begin{align*}
    \sum_{t=0}^T \mu \eta_t \Vert \vz_{t+1/2 } - \vz^* \Vert^2 \le \frac{1}{2} \Vert \vz_0 - \vz^* \Vert^2 - \frac{1}{2} \Vert \vz_{T} - \vz^* \Vert^2. 
\end{align*}
Using Jensen's inequality, for each epoch, we have
\begin{align*}
    \Vert \bar \vz_T - \vz^* \Vert^2 \le \frac{\Vert \vz_0 - \vz^* \Vert^2}{2 \mu \sum_{t=0}^{T-1} \eta_t} \le \frac{M \Vert \vz_0 - \vz^* \Vert^3}{\mu T^{3/2}} := c \Vert \vz_0 - \vz^* \Vert^2.
\end{align*}
Next, we consider the iterate $ \{\vz^{(s)} \}_{s=0}^{S-1}$. For the first epoch,
the setting of $T$ ensures $c \le 1/2$:
\begin{align*}
    \Vert \vz^{(1)} - \vz^* \Vert^2 \le \frac{1}{2} \Vert \vz_0 - \vz^* \Vert^2.
\end{align*}
Then for the second one, it is improved by
\begin{align*}
    \Vert \vz^{(2)} - \vz^* \Vert^2 \le \frac{\Vert \vz^{(1)} - \vz^* \Vert^3}{2 \Vert \vz_0 - \vz^* \Vert} \le \left( \frac{1}{2} \right)^{1 + 3/2} \Vert \vz_0 - \vz^* \Vert^2.
\end{align*}
Keep repeating this process. We can get
\begin{align*}
    \Vert \vz^{(s)} - \vz^* \Vert^2 \le  \left( \frac{1}{2} \right)^{q_s} \Vert \vz_0 - \vz^* \Vert^2,
\end{align*}
where $q_s$ satisfies the recursion
\begin{align*}
    q_{s} = 
    \begin{cases}
        1, & s = 1; \\[2mm]
        \dfrac{3}{2}q_{s-1} +1, & s \ge 2.
    \end{cases}
\end{align*}
This implies
\begin{align*}
    \Vert \vz^{(s)} - \vz^* \Vert^2 \le \left( \frac{1}{2}\right)^{\left(\dfrac{3}{2}\right)^{s-1}+1} \Vert \vz_0 - \vz^* \Vert^2.
\end{align*}
Set $m=\Theta(d)$, LEN-restart takes $\gO(d^{2/3}\kappa^{2/3}\log\log(1/\epsilon))$ oracle to $\mF(\cdot)$ and $\gO((1+d^{-1/3}\kappa^{2/3})\log\log(1/\epsilon))$ oracle to $\nabla \mF(\cdot)$. 
Under Assumption~\ref{asm:arith-cmp}, the computational complexities of the oracles is
\begin{align*}
    &\gO\left(N\cdot d^{2/3}\kappa^{2/3}\log\log(1/\epsilon) + Nd\cdot(1+d^{-1/3}\kappa^{2/3})\log\log(1/\epsilon)\right) \\
    &=\gO\left((Nd+Nd^{2/3}\kappa^{2/3})\log\log(1/\epsilon)\right).
\end{align*}

\end{proof}

\section{Proof of Corollary~\ref{thm:LEN-SCSC-complexity}}
\begin{proof}
The computational complexity of inner loop can be directly obtained by replacing $\epsilon^{-1}$ by $\kappa$ in Theorem~\ref{thm:LEN-CC-complexity} such that
\begin{align*}
 \text{ Inner Computational Complexity} =   \tilde{\gO}\left((N+d^2)\cdot(d+d^{2/3}\kappa^{2/3})\right).
\end{align*}
The iterations of outer loop is $S=\log\log(1/\epsilon)$, thus, the total computational complexity of LEN-restart is 
\begin{align*}
    S\cdot\text{ Inner Computational Complexity} = \tilde{\gO}\left((N+d^2)\cdot(d+d^{2/3}\kappa^{2/3})\right).
\end{align*}
\end{proof}

\section{Computational Complexity Using Fast Matrix Operations} \label{apx:fast-matrix}

Theoretically, one may use fast matrix operations for Schur decomposition and matrix inversion~\citep{demmel2007fast}, with a computational complexity of $d^{\omega}$, where $\omega \approx 2.371552$ is the matrix multiplication constant. In this case, the total computational complexity of Algorithm \ref{alg:LEN-imple} is
\begin{align*}
    \tilde \gO \left(\left(\frac{Nd + d^{\omega}}{m}  + d^2 + N \right) m^{2/3} \epsilon^{-2/3} \right)
\end{align*}
Setting the optimal $m$, we obtain the following complexity of Algorithm \ref{alg:LEN-imple}:
\begin{align*}
    \begin{cases}
        \tilde \gO(d^{\frac{2}{3}(\omega+1)} \epsilon^{-2/3}) ~~(\text{with } m= d^{\omega-2}) , &  N \lesssim d^{\omega-1} \\
        \tilde \gO(N^{2/3} d^{4/3} \epsilon^{-2/3}) ~~ (\text{with } m = N / d) , & d^{\omega-1} \lesssim N \lesssim d^2 \\
        \tilde \gO(N d^{2/3} \epsilon^{-2/3}) ~~ (\text{with } m=d), & d^2 \lesssim N.
    \end{cases}
\end{align*}
Our result is always better than the $\gO((N d+ d^{\omega}) \epsilon^{-2/3} )$ of existing optimal second-order methods.

\section{{The Inexact Algorithm}} \label{apx:inexact}

\begin{algorithm*}[t]  
\caption{Inexact LEN$(\vz_0, T, m,M, \alpha)$}  \label{alg:LEN-inexact}
\begin{algorithmic}[1] 
\STATE \textbf{for} $t=0,\cdots,T-1$ \textbf{do} \\
% \STATE \quad \textbf{if} $t \mod m =0$ \textbf{do} \\
% \STATE \quad \quad Update the snapshot $\vw_t = \vz_t$  and compute its Hessian $\nabla 
%  \mF(\vw_t) $ \\
% \STATE \quad \textbf{end if} \\
\STATE \quad Use Algorithm \ref{alg:line-search-eta} to find $(\vz_{t+1/2}, \gamma_t)$ that satisfies
\begin{align*}
    \vz_{t+1/2} = \vz_t - ( \nabla \mF({\vz_{\pi(t)}}) + \gamma_t \mI_d )^{-1} \mF(\vz_t)
\end{align*}
and {  $M\|\vz_t-\vz_{t+1/2}\| \le \gamma_t \le \alpha M \|\vz_t-\vz_{t+1/2}\|   $} for given $\alpha \ge 1$. \label{line:aux}
% \begin{align*}
%     \mF(\vz_t) = (\nabla \mF({\color{blue}\vz_{\pi(t)}}) + M \Vert \vz_t - \vz_{t+1/2} \Vert \mI_{d})  (\vz_t - \vz_{t+1/2}). 
% \end{align*}
\\
\STATE \quad Compute extra-gradient step $ \vz_{t+1} = \vz_t -  \gamma_t^{-1} \mF(\vz_{t+1/2}) $. 
% \label{line:extra}
\STATE \textbf{end for} \\
\STATE \textbf{return} $ \bar \vz_T = \frac{1}{\sum_{t=0}^{T-1} \gamma_t^{-1}} \sum_{t=0}^{T-1} \gamma_t^{-1} \vz_{t+1/2}$.
\end{algorithmic}
\end{algorithm*}

Algorithm \ref{alg:LEN} requires a cubic regularized Newton (CRN) oracle (Implicit Step, (\ref{eq:LEN-update})). We provide implementation details for the CRN oracle in Section \ref{sec:imple}. One missing detail is that 
we can not obtain the exact solution to the CRN oracle in practice.
To make our result more rigorous, we analyze the inexact LEN (Algorithm \ref{alg:LEN}), which allows inexact sub-problem solving with a parameter $\alpha\ge 1$. Note that this algorithm reduces to the exact version (Algorithm \ref{alg:LEN}) when $\alpha = 1$.

Below, we present the following theorem as the inexact version of Theorem \ref{thm:LEN}.
\begin{thm} \label{thm:LEN-inexact}
    Suppose that Assumption \ref{asm:prob-lip-hes} and \ref{asm:prob-cc} hold. Let $\vz^* = (\vx^*,\vy^*)$ be a saddle point and $\beta = \Vert \vz_0 - \vz^* \Vert$. Set $M \ge 3 \rho m$. The sequence of iterates generated by Algorithm \ref{alg:LEN-inexact} is bounded $\vz_t \in \sB_{\beta}(\vz^*), ~ \vz_{t+1/2} \in \sB_{3\beta}(\vz^*) , \quad \forall t = 0,\cdots,T-1,$
    % \begin{align*}
    %     \vz_t, ~ \vz_{t+1/2} \in \sB_{\beta}(\vz^*) , \quad \forall t = 0,\cdots,T-1,
    % \end{align*}
    and satisfies the following ergodic convergence:
    \begin{align*}
        {\rm Gap}(\bar \vx_T, \bar \vy_T; 3\beta) \le \frac{16 \alpha M \Vert \vz_0 - \vz^* \Vert^3}{T^{3/2}}.
    \end{align*}
Let $M = 3 \rho m $ and $\alpha=2$. Algorithm \ref{alg:LEN} finds an $\epsilon$-saddle point within $\gO(m^{2/3} \epsilon^{-2/3} )$ iterations.
% Under Assumption \ref{asm:arith-cmp}, Algorithm \ref{alg:LEN} with $m = \Theta(d) $ takes the total complexity of $\gO( N d + N d^{2/3} \epsilon^{-2/3})$ to call the oracles.
\end{thm}

\begin{proof}
    See Section \ref{apx:proof-main}.
\end{proof}
% In fact, Section \ref{apx:proof-main} has proven that Algorithm \ref{alg:LEN-inexact} has the same convergence rate as the exact algorithm when treating $\alpha$ as a constant. For completeness,

The only remaining thing is to show how to compute $\gamma_t$ in the auxiliary problem (Line \ref{line:aux} in Algorithm \ref{alg:LEN-inexact}). Below, we present an efficient sub-procedure to achieve the desired goal using the 
standard Newton step. We
define the monotone operator $\mA_t: \sR^d \rightarrow \sR^d$ by 
\begin{align} \label{eq:At}
    \mA_t(\vz) = \mF(\vz_t) + \nabla \mF(\vz_{\pi(t)}) (\vz - \vz_t).
\end{align}
Then we can write down the (regularized) Newton step as
\begin{align} \label{eq:Newton-step}
\begin{split}
      \vz_{t+1/2}(\eta;\vz_t) &:= \vz_t - ( \nabla \mF(\vz_{\pi(t)}) + \eta^{-1} \mI_d )^{-1} \mF(\vz_t) \\
    &= (\mI_d + \eta \mA_t)^{-1} (\vz_t).
\end{split}
\end{align}
And the inexact condition (Line \ref{line:aux} in Algorithm \ref{alg:LEN-inexact}) is
\begin{align} \label{eq:inexact-eta}
    \frac{1}{\alpha M} \le \phi_t(\eta; \vz_t) \le \frac{1}{M},
\end{align}
where $\phi_t(\eta;\vz_t)$ is defined as $\phi_t(\eta; \vz_t):= \eta \Vert \vz_{t+1/2}(\eta; \vz_t) -  \vz_t \Vert $. 

Note that a stepsize $\eta$ that satisfies (\ref{eq:inexact-eta}) directly implies $\gamma_t = 1/ \eta$ satisfies the requirement of Line \ref{line:aux} in Algorithm \ref{alg:LEN-inexact}. Therefore, the main goal of this section is to design a sub-procedure that can determine the stepsize $\eta$ that satisfies (\ref{eq:inexact-eta}).

A similar sub-procedure without using lazy Hessian updates has been proposed in \citep{monteiro2012iteration}. Below, we show that we can use a similar sub-procedure for our algorithm.  We recall some useful lemmas in \citep{monteiro2012iteration}, which holds for any monotone operators $\mA$. Below, we state their results when $\mA = \mA_t$.

\begin{lem}[Lemma 4.3 and Lemma 4.4 \citep{monteiro2012iteration}] \label{lem:phi}
Recall the definition of $\phi_t$ right after (\ref{eq:inexact-eta}). For any $\vz \in \sR^d$, the following statements hold:
\begin{enumerate}

        \item For any $\eta>0$, we have $\phi_t(\eta; \vz)>0$.
        %and $\phi_t(\eta;\vz)$ is continuous at $\eta$.
        \item For any $ 0< \eta' \le \eta$, we have that
        \begin{align*}
            \frac{\eta}{\eta'} \phi_t(\eta'; \vz) \le \phi_t(\eta; \vz) \le \left(    \frac{\eta}{\eta'}  \right)^2 \phi_t(\eta'; \vz).
        \end{align*}
        As a corollary, $\phi_t(\eta; \vz )$ is a continuous and strictly increasing function, which converges to $0$ or $+\infty$ as $\eta$ tends to $0$ or $+\infty$, respectively.
        \item For any $0<\beta^- < \beta^+$, the set of all scalars $\eta>0$ satisfying $ \beta^- \le \phi_t(\eta;\vz) \le \beta^+$ is a closed interval $[\eta^-,\eta^+]$ such that $ \eta^+/ \eta^- \ge \sqrt{\beta^+/\beta^-} $.
    \end{enumerate}
\end{lem}

Algorithm \ref{alg:line-search-eta} presents our sub-procedure to output the tuple $(\vz_{t+1/2}, \gamma_t)$ satisfying (\ref{eq:inexact-eta}). Similar to \citep{monteiro2012iteration}, the procedure consists of two stages. The first one is a bracketing stage, which either outputs an acceptable solution or an initial interval $[c_t^-,c_t^+]$ that contains all the $\eta$ satisfying (\ref{eq:inexact-eta}). The second one is a bisection stage, which uses binary search in the logarithmic scale to find a stepsize $\eta$ satisfying (\ref{eq:inexact-eta}). Note that the log-scale binary search would finally lead to a $\gO(\log \log (1/\epsilon))$ iteration complexity, which improves the $\gO(\log (1/\epsilon))$ iteration complexity using naive binary search in \citep{adil2022optimal,bullins2022higher}.

% Newton step $\vz_{t+1/2} = $ can be written as $$. 

\begin{algorithm*}[t]  
\caption{Bracketing/Bisection Procedure$(\mA_t, \vz_t, M, \alpha, \eta_t^0)$}  \label{alg:line-search-eta}
\begin{algorithmic}[1] 
\STATE \textbf{(Bracketing Stage)} Compute $\vz_{t+1/2}^0 = (\mI_d + \eta_t^0 \mA_t)^{-1} (\vz_t)$ with one Newton step. \\
\quad (1a) \textbf{if} $\eta_t^0 \Vert \vz_{t+1/2}^0 - \vz_t \Vert \in (\frac{1}{\alpha M}, \frac{1}{M})$, \textbf{then} let $\vz_{t+1/2} = \vz_{t+1/2}^0, \eta_t = \eta_t^0$ and \textbf{go to} Line \ref{line:return}. \\
\quad (1b) \textbf{if} $\eta_t^0 \Vert \vz_{t+1/2}^0 - \vz_t \Vert < \frac{1}{\alpha M}$, \textbf{then} set $c_t^- = \eta_t^0$ and $c_t^+ = \frac{1}{M \Vert \vz_{t+1/2}^0 - \vz_t \Vert }$; \\
\quad (1c) \textbf{if} $\eta_0^t \Vert \vz_{t+1/2}^0 - \vz_t \Vert > \frac{1}{M} $, \textbf{then} set $c_t^- = \frac{1}{\alpha M \Vert \vz_{t+1/2}^0 - \vz_t \Vert } $ and $c_t^+ = \eta_t^0$; \\
\STATE \textbf{(Bisection Stage)} \\
\quad (2a) set $\eta_t = \sqrt{c_t^- c_t^+}$ and compute $\vz_{t+1/2} = (\mI_d + \eta_t \mA_t)^{-1} (\vz_t)$ with one Newton step; \\
\quad (2b) \textbf{if} $\eta_t \Vert \vz_{t+1/2} - \vz_t \Vert \in (\frac{1}{\alpha M}, \frac{1}{M})$, \textbf{then} \textbf{go to} Line \ref{line:return}; \\
\quad (2c) \textbf{if}  $\eta_t \Vert \vz_{t+1/2} - \vz_t \Vert > \frac{1}{M}$, \textbf{then} set $c_t^+ = \eta_t$; \textbf{else} set $c_t^- = \eta_t$; \\
\quad (2d) \textbf{go to} step (2a). \\
\STATE \textbf{return} $(\vz_{t+1/2}, \gamma_t)$ that meets the requirement of Line \ref{line:aux} in Algorithm \ref{alg:LEN-inexact}, where $\gamma_t  =1/\eta_t$ \label{line:return}.
\end{algorithmic}
\end{algorithm*}

Our first result of Algorithm \ref{alg:line-search-eta} is the correctness of the bracketing stage, stated as follows.

\begin{lem}
Let $ [\eta_t^-,\eta_t^+]$ be the interval that contains all the stepsizes satisfying (\ref{eq:inexact-eta}). Compute $\vz_{t+1/2}^0 = (\mI_d + \eta_t^0 \mA_t)^{-1} (\vz_t)$ with one Newton step as Algorithm \ref{alg:line-search-eta}. The following statements hold:
\begin{enumerate}
    \item if $\eta_t^0 \Vert \vz_{t+1/2}^0 - \vz_t \Vert < \frac{1}{\alpha M}$, then $\eta_t^0 < \eta_t^-$ and  $\eta_t^+ \le \frac{1}{M \Vert \vz_{t+1/2}^0 - \vz_t \Vert}$;
    \item if $\eta_t^0 \Vert \vz_{t+1/2}^0 - \vz_t \Vert > \frac{1}{M}$, then $\eta_t^+ < \eta_t^0$ and $\frac{1}{\alpha M \Vert \vz_{t+1/2}^0 - \vz_t \Vert} \le \eta_t^-$.
\end{enumerate}
\end{lem}

\begin{proof}
    We only prove the first claim since the proof of the second claim follows in a similar manner.

Recall the definition of $\phi_t$ right after (\ref{eq:inexact-eta}). The condition $\eta_t^0 \Vert \vz_{t+1/2}^0 - \vz_t \Vert < \frac{1}{\alpha M}$ is equivalent to $ \phi_t(\eta_{t}^0; \vz_t) < \phi_t(\eta_t^-; \vz_t)$. Firstly. the fact that $\phi_t(\eta_t^-;\vz_t)$ is a strictly increasing function according to the second statement in Lemma \ref{lem:phi}, we know that $\eta_t^0 < \eta_t^-$.

Secondly, using the inequality in the second statement of Lemma \ref{lem:phi}, we know that
\begin{align*}
    \eta_t^+ \Vert \vz_{t+1/2}^0 - \vz_t \Vert = \frac{\eta_t^+}{\eta_t^0} \phi_t(\eta_t^0; \vz_t) \le \phi_t(\eta_t^+; \vz_t) = \frac{1}{M},
\end{align*}
which implies $\eta_t^+ \le \frac{1}{M \Vert \vz_{t+1/2}^0 - \vz_t \Vert}$ by rearranging.

\end{proof}

Therefore, the bracketing stage can always output an interval that contains the acceptable stepsizes $\eta$ satisfying (\ref{eq:inexact-eta}). Given such a valid initial interval, the bisection stage always find an acceptable stepsize, stated as follows.

\begin{lem} \label{lem:comp-bisec}
Consider Algorithm \ref{alg:line-search-eta}. If the bracketing stage outputs an interval $[c_t^-,c_t^+]$ containing all the stepsizes $\eta$ satisfying (\ref{eq:inexact-eta}), which is then input to the bisection stage, then the number of Newton step during the bisection stage is bounded by $ 1 + \log (\log (h_t) / \log \alpha ))$, where 
\begin{align} \label{eq:ht}
    h_t = \max \left\{ \frac{1}{\eta_t^0 M  \Vert \vz_{t+1/2}^{0} - \vz_t \Vert }, ~ \alpha M \eta_t^0 \Vert \vz_{t+1/2}^{0} - \vz_t \Vert   \right\}
\end{align}
is the maximal ratio of $c_t^+ / c_t^-$.
\end{lem}

\begin{proof}
   After $j$ steps of bisection iterations, we have that $ \log \frac{c_t^+}{c_t^-}  = \frac{1}{2^j} \log h_t $.  In view of  the third statement in Lemma \ref{lem:phi}, we know that $c_t^+/  c_t^- \ge \sqrt{\alpha}$. These two inequalities immediately imply that the bisection stage would terminates in $ j \le 1 + \log (\log (h_t) / \log \alpha )) $ iterations.
\end{proof}

Our goal from now on would be giving a uniform upper bound of $h_t$ all for $t$, which can imply the total complexity of our algorithm. From the definition of $h_t$ in (\ref{eq:ht}), we need to give both lower and upper bounds of $\eta_t^0 \Vert \vz_{t+1/2}^0 - \vz_t \Vert$. We recall some technical lemmas in \citep{monteiro2012iteration}.

\begin{lem}[Proposition 4.5 \citet{monteiro2012iteration}] \label{lem:A-nonexpansive}
Let $\mA: \sR^d \rightarrow \sR^d$ be a monotone operator. For a point $\vz^* \in \sR^d$ such that $\mA(\vz^*) = 0$, for any $\eta>0$ and $\vz \in \sR^d$ it holds that
\begin{align*}
    \max \left\{ \Vert (\mI_d + \eta \mA)^{-1} \vz -  \vz^* \Vert,  \Vert (\mI_d + \eta \mA)^{-1} \vz -  \vz \Vert   \right\} \le \Vert \vz - \vz^* \Vert.
\end{align*}
\end{lem}

% first prove a technical lemma, which generalize Proposition 4.9 in \citep{monteiro2012iteration}
% with lazy Hessian updates.

% \begin{lem}[] \label{lem:tech-prop49} Recall the definition of $\phi_t$ right after~(\ref{eq:inexact-eta}). Suppose that Assumptions \ref{asm:prob-cc} and \ref{asm:prob-lip-hes} hold. Let $\vz^* = (\vx^*,\vy^*)$ be the
% unique saddle point. Then,
% \begin{align*}
%     \phi_t(\eta;\vz_t) \le \eta \Vert \
% \end{align*}
  
% \end{lem}

From now on, we will fix all the $\eta_t^0$ in all the iterations such that $\eta_t^0 = \bar \eta $ and analyze Algorithm \ref{alg:LEN-inexact}. The following lemma shows a uniform upper bound of  $ \Vert \vz_{t+1/2}^0 - \vz_t \Vert$.

\begin{lem}[Upper bound of  $ \Vert \vz_{t+1/2}^0 - \vz_t \Vert$] \label{lem:sub-upper}
Suppose that Assumption \ref{asm:prob-lip-hes} and \ref{asm:prob-cc} hold. Let $\vz^* = (\vx^*,\vy^*)$ be a saddle point. Set $M = 3 \rho m$ as in Theorem \ref{thm:LEN}. For all the iterations of Algorithm \ref{alg:LEN-inexact}, it holds that
\begin{align} \label{eq:upper-bound}
     \Vert \vz_{t+1/2}^0 - \vz_t \Vert \le   \Vert \vz_0 - \vz^* \Vert + \frac{5  \bar \eta \rho}{2} \Vert \vz_0 - \vz^* \Vert^2.
\end{align}
\end{lem}

\begin{proof}
    Let $r_t:= \mF(\vz^*) - \mA_t(\vz^*)$ and define the operator $\tilde \mA_t$ as $\tilde \mA_t(\vz) = \mA_t(\vz) + r_t$. From the definition of $\tilde \mA_t$ we know that all any $\eta>0$ and $\vz \in \sR^d$ we have that
    \begin{align} \label{eq:AA}
        (\mI_d +  \eta \tilde \mA_t )^{-1} (\vz +  \eta r_t) = (\mI_d + \eta \mA_t)^{-1} (\vz)
    \end{align}
    Now we upper bound  $\Vert \vz_{t+1/2}^0 - \vz_t \Vert$ as follows.
    \begin{align} \label{eq:r-ast}
    \begin{split}
          &\quad  \Vert \vz_{t+1/2}^0 - \vz_t \Vert \\
       &= \Vert (\mI_d + \bar \eta \mA_t)^{-1} (\vz_t) - \vz_t \Vert \\
       &= 
        \Vert (\mI_d +  \bar \eta \tilde \mA_t )^{-1} (\vz_t +  \bar \eta r_t) - \vz_t \Vert  \\
       &\le  \Vert (\mI_d +  \bar \eta \tilde \mA_t )^{-1} (\vz_t) - \vz_t \Vert +  \Vert (\mI_d +  \bar \eta \tilde \mA_t )^{-1} (\vz_t) -  (\mI_d +  \bar \eta \tilde \mA_t )^{-1} (\vz_t +  \bar \eta r_t) \Vert \\
       &\le\Vert \vz_t - \vz^* \Vert + \bar \eta \Vert r_t \Vert,
    \end{split}
    \end{align}
    where in the last step we use Lemma \ref{lem:A-nonexpansive} to upper bound the first term and use the non-expansiveness of resolvent (see \textit{i.e.} \citep{rockafellar1976monotone}) to upper bound the second term. 

    We continue to upper bound $\Vert r_t \Vert$. Recall the definition of $\mA_t$ in (\ref{eq:At}), we know that
    \begin{align*}
        r_t &= \mF(\vz^*) - \mF(\vz_t) - \nabla \mF(\vz_{\pi(t)}) (\vz^*  - \vz_t) \\
        &= \mF(\vz^*) - \mF(\vz_t) - \nabla \mF(\vz_{t}) (\vz^*  - \vz_t) + (\nabla \mF(\vz_{t}) - \nabla \mF(\vz_{\pi(t)})  (\vz^*  - \vz_t)
    \end{align*}
    Note that $\nabla \mF$ is $\rho$-Lipschitz continuous.
Taking norm on both sides of the above identity, we have 
\begin{align*}
    \Vert r_t \Vert \le \frac{\rho}{2} \Vert \vz^* - \vz_t \Vert^2 + \rho \Vert \vz_t - \vz_{\pi(t)} \Vert \Vert \vz^* - \vz_t \Vert
\end{align*}
Recalling (\ref{eq:iterate-bounded}) that we have $\Vert \vz_t - \vz^* \Vert \le \Vert \vz_0 - \vz^* \Vert$ for all $t$, by the triangle inequality we also have $\Vert \vz_t - \vz_{\pi(t)} \Vert \le 2 \Vert \vz_0 - \vz^* \Vert$. Therefore, we have that
$\Vert r_t \Vert \le  \frac{5}{2} \Vert \vz_0 - \vz^* \Vert^2$. Finally, we plug into (\ref{eq:r-ast}) to obtain the desired upper bound in (\ref{eq:upper-bound}).

\end{proof}

Next, we give a uniform lower bound of  $\Vert \vz_{t+1/2}^0 - \vz_t \Vert$.

\begin{lem}[Lower bound of  $\Vert \vz_{t+1/2}^0 - \vz_t \Vert$] \label{lem:sub-lower}
Suppose that Assumption \ref{asm:prob-lip-hes} and \ref{asm:prob-cc} hold. Let $\vz^* = (\vx^*,\vy^*)$ be a saddle point and $\beta = \Vert \vz_0 - \vz^* \Vert$. Set $M = 3 \rho m$ as in Theorem \ref{thm:LEN}. If in all the iterations of Algorithm \ref{alg:line-search-eta} the point $\vz_{t+1/2}^0$ is not an $\epsilon$-solution, it holds that
\begin{align} \label{eq:lower-bound}
    \bar \eta \Vert \vz_{t+1/2}^0 - \vz_t \Vert \ge \xi_t, 
\end{align}
where $\xi_t = \min \left\{2  \beta, \frac{\bar \eta \epsilon}{8 \beta (3 \bar \eta \beta \rho + 1)}  \right\}$.
\end{lem}

\begin{proof} 
 We show a  contradiction if (\ref{eq:lower-bound}) does not hold.
Firstly, if $\vz_{t+1/2}^0 = (\vx_{t+1/2}^0, \vy_{t+1/2}^0)$ is not an $\epsilon$-solution to the problem, then by Lemma \ref{lem:gap-gap} and \ref{lem:avg} we know that $\Vert \mF(\vz_{t+1/2}^0) \Vert $ must be large:
    \begin{align*}
        \epsilon \le {\rm Gap}(\vx_{t+1/2}^0, \vy_{t+1/2}^0; 3 \beta) &\le \max_{\vz \in \sB_{3 \sqrt{2} \beta}(\vz^*)} \langle \mF(\vz_{t+1/2}^0), \vz_{t+1/2}^0 - \vz \rangle \le 8\beta \Vert \mF(\vz_{t+1/2}^0) \Vert, 
    \end{align*}
    where the last step uses that $\Vert \vz_{t+1/2}^0 - \vz_t \Vert \le 2 \beta $ if (\ref{eq:lower-bound}) does not hold, $\Vert \vz_t - \vz^* \Vert \le \beta $ and the triangle inequality. Therefore, we can conclude that 
    \begin{align} \label{eq:large-gnorm}
         \Vert \mF(\vz_{t+1/2}^0) \Vert \ge \frac{\epsilon}{8 \beta}.
    \end{align}
Secondly, from the update of the algorithm, we have that
\begin{align*}
    \vz_t - \vz_{t+1/2}^0 = \bar \eta ( \mF(\vz_t) + \nabla \mF(\vz_{\pi(t)} ) (\vz_{t+1/2}^0 - \vz_t))
\end{align*}
Then we further know that
\begin{align*}
    &\quad \vz_t - \vz_{t+1/2}^0 - \mF(\vz_{t+1/2}^0) \\
    &= \bar \eta ( \mF(\vz_t) + \nabla \mF(\vz_{\pi(t)} ) (\vz_{t+1/2}^0 - \vz_t) - \mF(\vz_{t+1/2}^0)) \\
    &= \bar \eta ( \mF(\vz_t) + \nabla \mF(\vz_{t} ) (\vz_{t+1/2}^0 - \vz_t) - \mF(\vz_{t+1/2}^0)) \\
    &\quad + \bar \eta (\nabla \mF(\vz_{t}) - \nabla \mF(\vz_{\pi(t)} )  ) (\vz_{t+1/2}^0 - \vz_t).
\end{align*}
  Note that $\nabla \mF$ is $\rho$-Lipschitz continuous.
Taking norm on both sides of the above identity, we have 
\begin{align*}
     &\quad \Vert \vz_t - \vz_{t+1/2}^0 - \mF(\vz_{t+1/2}^0)  \Vert \\
     &\le \frac{\bar \eta \rho }{2} \Vert \vz_{t+1/2}^0 - \vz_t \Vert^2 + \bar \eta \rho \Vert \vz_t - \vz_{\pi(t)} \Vert \Vert \vz_{t+1/2}^0 - \vz_t \Vert \\
     &\le 3 \bar \eta  \beta \rho \Vert \vz_{t+1/2}^0 - \vz_t \Vert.
\end{align*}
where the last step uses the triangle inequality, that $\Vert \vz_{t+1/2}^0 - \vz_t \Vert \le 2 \beta $ if (\ref{eq:lower-bound}) does not hold, and that  $\Vert \vz_t - \vz^* \Vert \le \Vert \vz_0 - \vz^* \Vert$ by
(\ref{eq:iterate-bounded}).
Then we can know that
    \begin{align*}
        \bar \eta \Vert \mF(\vz_{t+1/2}^0) \Vert &\le  \Vert   \vz_t - \vz_{t+1/2}^0 - \bar \eta  \mF(\vz_{t+1/2}^0) \Vert +  \Vert \vz_{t+1/2}^0 - \vz_t \Vert \\
        &\le (3 \bar \eta \beta \rho + 1) \Vert \vz_{t+1/2}^0 - \vz_t \Vert.
    \end{align*}
    Recalling (\ref{eq:large-gnorm}), we know that this would contradict the hypothesis that (\ref{eq:lower-bound}) does not hold.
\end{proof}

 Lemma \ref{lem:sub-upper} and Lemma \ref{lem:sub-lower} tell us that the $h_t$ defined in (\ref{eq:ht}) is uniformly bounded for all $t$. Finally, we obtain the following theorem by combining Theorem \ref{thm:LEN-inexact} and Theorem \ref{lem:comp-bisec}.

 \begin{thm}
Suppose that Assumption \ref{asm:prob-lip-hes} and \ref{asm:prob-cc} hold. Let $\vz^* = (\vx^*,\vy^*)$ be a saddle point and $\beta = \Vert \vz_0 - \vz^* \Vert$. Set $M \ge 3 \rho m$. The sequence of iterates generated by Algorithm \ref{alg:LEN-inexact} is bounded $\vz_t \in \sB_{\beta}(\vz^*), ~ \vz_{t+1/2} \in \sB_{3\beta}(\vz^*) , \quad \forall t = 0,\cdots,T-1,$
    % \begin{align*}
    %     \vz_t, ~ \vz_{t+1/2} \in \sB_{\beta}(\vz^*) , \quad \forall t = 0,\cdots,T-1,
    % \end{align*}
    and satisfies the following ergodic convergence:
    \begin{align*}
        {\rm Gap}(\bar \vx_T, \bar \vy_T; 3\beta) \le \frac{16 \alpha M \Vert \vz_0 - \vz^* \Vert^3}{T^{3/2}}.
    \end{align*}
Let $M = 3 \rho m $ and $\alpha=2$. Algorithm \ref{alg:LEN} finds an $\epsilon$-saddle point within $\gO(m^{2/3} \epsilon^{-2/3} )$ iterations.

If we call the sub-procedure (Algorithm \ref{alg:line-search-eta}) with fixed $\eta_t^0 = \bar \eta$, every call of this sub-procedure makes at most $ \gO( \log \log ({\rm poly}(m, \beta, \rho, \bar \eta, 1/\epsilon))$ Newton steps.
 \end{thm}

 The above theorem shows that the CRN sub-problem can be solved to guarantee the desired precision for target problem in $\gO(\log \log (1/\epsilon))$ iterations, which tightens the $\gO(\log (1/\epsilon))$ iteration complexity in \citep{bullins2022higher,adil2022optimal}. Additionally,  \citep{bullins2022higher,adil2022optimal} requires additionally assume ${\sigma_{\min}}( \nabla \mF(\vz)) \ge \mu$ for some positive constant $\mu$, which makes the problem similar to strongly-convex(-strongly-concave) problems, while our analysis does not require such an assumption.

\end{document}